\documentclass{amsart}
\usepackage{amssymb,latexsym}
\usepackage{enumerate}
\usepackage[all,cmtip]{xy}
\usepackage{color}
\usepackage{float}
\usepackage{varioref}
\usepackage[bookmarksnumbered,colorlinks,plainpages]{hyperref}
\usepackage{cite}

\newtheorem{thm}{Theorem}[section]
\newtheorem{cor}[thm]{Corollary}
\newtheorem{lem}[thm]{Lemma}

\newtheorem{prop}[thm]{Proposition} 

\theoremstyle{definition}

\newtheorem{exa}[thm]{Example}

\numberwithin{equation}{section}

\DeclareMathOperator{\vect}{vect}
\newcommand{\rperp}[1]{#1^{\perp}}

\newcommand{\Dd}{\mathcal{D}}
\newcommand{\Hh}{\mathcal{H}}
\newcommand{\Tt}{\mathcal{T}}

\newcommand{\bp}{\bar{p}}
\newcommand{\bq}{\bar{q}}
\newcommand{\bh}{\bar{h}}
\newcommand{\bT}{\bar{T}}
\newcommand{\sT}{T^\ast}

\newcommand{\pp}{\textrm{\ul{p}}}
\newcommand{\lala}{\underline{\lambda}}
\newcommand{\cohnull}[1]{\mathrm{coh}_0\,#1}

\newcommand{\LF}[2]{\langle #1,#2\rangle}
\newcommand{\DLF}[2]{\langle\!\langle #1,#2\rangle\!\rangle}

\newcommand{\Der}{\operatorname{D^b}}

\newcommand{\Hom}{\operatorname{Hom}}
\newcommand{\End}{\operatorname{End}}
\newcommand{\Ext}{\operatorname{Ext}}

\newcommand{\dual}{\operatorname{D}}
\newcommand{\gen}[1]{\langle #1\rangle}
\newcommand{\eul}[1]{\langle #1\rangle}
\newcommand{\coh}{\operatorname{coh}}
\newcommand{\Qcoh}[1]{\operatorname{Qcoh}\,{#1}}
\newcommand{\Groth}{\operatorname{K_0}}
\newcommand{\Pic}{\operatorname{Pic}}
\newcommand{\can}{\mathrm{can}}
\newcommand{\her}{\mathrm{her}}
\newcommand{\cox}{\mathrm{cox}}
\newcommand{\squid}{\mathrm{squid}}
\newcommand{\width}{\operatorname{w}}

\renewcommand{\AA}{\mathbb{A}}

\newcommand{\EE}{\mathbb{E}}
\newcommand{\LL}{\mathbb{L}}
\newcommand{\XX}{\mathbb{X}}
\newcommand{\YY}{\mathbb{Y}}
\newcommand{\ZZ}{\mathbb{Z}}
\newcommand{\NN}{\mathbb{N}}

\newcommand{\vc}{{\vec{c}}}
\newcommand{\vu}{\vec{u}}

\newcommand{\vom}{\vec{\omega}}
\newcommand{\vx}{{\vec{x}}}
\newcommand{\vy}{{\vec{y}}}

\newcommand{\ovp}{\overline{p}}
\newcommand{\lcm}{\operatorname{lcm}}

\newcommand{\incl}{\hookrightarrow}

\newcommand{\Ocr}{\mathcal{O}}

\newcommand{\rk}{\operatorname{rk}}
\newcommand{\class}[1]{\left[#1\right]}

\newcommand{\HVCenter}[1]{\setbox 0=\hbox{#1}%
        \dimen0=\wd0%
        \dimen1=\ht0%
        \divide\dimen0 by 2%
        \divide\dimen1 by 2%
        \hskip -\dimen0%
        \lower \dimen1%
        \box0%
        \hskip -\dimen0}
\newcommand{\HBCenter}[1]{\setbox 0=\hbox{#1}%
        \dimen0=\wd0%
        \dimen1=\ht0%
        \divide\dimen0 by 2%
        \hskip -\dimen0%
        \box0%
        \hskip -\dimen0}
\newcommand{\HTCenter}[1]{\setbox 0=\hbox{#1}%
        \dimen0=\wd0%
        \dimen1=\ht0%
        \divide\dimen0 by 2%
        \hskip -\dimen0%
        \lower \dimen1%
        \box0%
        \hskip -\dimen0}
\newcommand{\RTCenter}[1]{\setbox 0=\hbox{#1}%
        \dimen0=\wd0%
        \dimen1=\ht0%
        \hskip -\dimen0%
        \lower \dimen1%
        \box0%
        \hskip -\dimen0}
\newcommand{\LTCenter}[1]{\setbox 0=\hbox{#1}%
        \dimen0=\wd0%
        \dimen1=\ht0%
        \lower \dimen1%
        \box0%
        \hskip -\dimen0}
\newcommand{\RVCenter}[1]{\setbox 0=\hbox{#1}%
        \dimen0=\wd0%
        \dimen1=\ht0%
        \divide\dimen1 by 2%
        \hskip -\dimen0%
        \lower \dimen1%
        \box0%
        \hskip -\dimen0}
\newcommand{\RBCenter}[1]{\setbox 0=\hbox{#1}%
        \dimen0=\wd0%
        \dimen1=\ht0%
        \hskip -\dimen0%
        \box0%
        \hskip -\dimen0}
\newcommand{\LVCenter}[1]{\setbox 0=\hbox{#1}%
        \dimen1=\ht0%
        \divide\dimen1 by 2%
        \lower \dimen1%
        \box0%
        \hskip -\dimen0}

\newcommand{\spitz}[1]{\langle #1\rangle}

\newcommand{\si}{\sigma}
\newcommand{\la}{\lambda}
\newcommand{\La}{\Lambda}
\newcommand{\ul}[1]{\underline{#1}}
\newcommand{\ra}{\rightarrow}
\newcommand{\lra}{\longrightarrow}
\renewcommand{\mod}{\operatorname{mod}}
\newcommand{\Mod}{\operatorname{Mod}}
\newcommand{\Ker}{\operatorname{ker}}
\newcommand{\Oo}{\mathcal{O}}
\newcommand{\sz}[1]{{\langle #1 \rangle}}

\newcommand{\bs}{\textbf{s}}
\newcommand{\bw}{\mathbf{w}}

\newcommand{\bz}{\textbf{z}}
\newcommand{\al}{\alpha}
\newcommand{\be}{\beta}
\newcommand{\ka}{\kappa}
\newcommand{\epi}{\twoheadrightarrow}
\newcommand{\Knull}{\mathrm{K}_0}
\newcommand{\QQ}{\mathbb{Q}}

\begin{document}

\title[Concealed-canonical algebras]{Extremal properties for concealed-canonical algebras}

\author[M. Barot]{Michael Barot}
\address{Instituto de Matem\'aticas, Universidad Nacional Aut\'onoma
de M\'exico, Ciudad Universitaria, C.P. 04510, Distrito Federal, Mexico}
\email{barot@matem.unam.mx}

\author[D. Kussin]{Dirk Kussin}
\address{Dipartimento di Informatica, Settore di Matematica, Universit\`{a} degli Studi di
  Verona, 37134 Verona, Italy}
\email{dirk@math.uni-paderborn.de}

\author[H. Lenzing]{Helmut Lenzing}
\address{Institut f\"{u}r Mathematik, Universit\"at Paderborn, 33095
  Paderborn, Germany}
\email{helmut@math.uni-paderborn.de}

\date{}

\begin{abstract}
Canonical algebras, introduced by C.M.~Ringel in 1984, play an
  important role in the representation theory of finite dimensional
  algebras. They are equipped with a large contact surface to many
  further mathematical subjects like function theory, 3-manifolds,
  singularity theory, commutative algebra, algebraic geometry and mathematical physics. We
  show in this paper that canonical algebras are characterized by a
  number of interesting extremal properties (among
  concealed-canonical algebras, that is, endomorphism rings of tilting bundles on a weighted projective line).
   We also investigate the corresponding class of algebras antipodal to
  canonical ones.  Our study yields new insight in the nature of concealed-canonical algebras, and sheds new light on an old question: Why are the
   canonical algebras canonical?
\end{abstract}


\maketitle

\section{Introduction}\label{sect:intro}

Canonical algebras were introduced by C.~M.\ Ringel in 1984
\cite{Ringel:1984} in order to solve an intriguing problem concerning
the representation type of a certain class of finite dimensional
algebras, now called tubular. When introducing weighted projective
lines in 1987 \cite{Geigle:Lenzing:1987}, Geigle and Lenzing showed
that canonical algebras arise as endomorphism rings of naturally
formed tilting bundles, consisting of line bundles. Due to this fact,
canonical algebras own a large contact surface to other parts of mathematics,
classical and modern.

Indeed, the \emph{canonical relations}
\begin{equation}
\label{eq:can_rel}
x_i^{p_i}=x_2^{p_2}-\la_ix_1^{p_i},\quad i=3,\ldots, t
\end{equation} defining the canonical algebras
already appeared 1882, respectively 1884, in the work of
H.\ Poincar\'{e}~\cite[p.~237]{Poincare:1882} (see also~\cite[p.\
183]{Poincare:Stillwell:1985}) and  F.\ Klein~\cite{Klein:1884}
yielding a link to Fuchsian singularities, respectively Kleinian (i.e.\ simple)
singularities. For modern accounts on this aspect we refer to work of J.\
Milnor~\cite{Milnor:1975} and W.\ D.\ Neumann~\cite{Neumann:1977}, compare
also~\cite{Lenzing:1994}.

An important feature of the canonical relations is the (graded)
factoriality of the commutative $k$-algebra
$S(\pp,\lala)=k[x_1,\ldots,x_t]/I$, where the ideal $I$ is
generated by the canonical relations~\eqref{eq:can_rel}. Moreover, by Mori~\cite{Mori:1977} and
Kussin~\cite{Kussin:1998} graded factoriality determines the algebras
$S(\pp,\lala)$ uniquely (among the
affine algebras of Krull dimension
two).

Canonical algebras belong to the larger class of \emph{concealed-canonical algebras}, see~\cite{Lenzing:Pena:1999,Lenzing:Meltzer:1996}, a class containing the tame concealed algebras, see~\cite{Happel:Vossieck:1983} and~\cite{Ringel:1984}. Concealed-canonical algebras may be defined as endomorphism rings of tilting bundles on a weighted projective line. By \cite{Skowronski:1996} and \cite{Lenzing:Pena:1999}, they are also characterized by the existence of a separating tubular family of sincere, standard stable tubes. Though the concepts now exist for many years, canonical and concealed-canonical algebras continue to be a topic of much current research. We just mention their appearance in recent papers dealing with the following subjects:
\begin{itemize}
\item[$\circ$] the theory of finite dimensional selfinjective algebras~\cite{Karpicz:Skowronski:Yamagata:2011},
\item[$\circ$] the invariant theory of module varieties~ \cite{Bobinski:2008b}, \cite{Bobinski:2008}, \cite{Bobinski:2008a},
\item[$\circ$] explicit matrix representations for exceptional modules, see \cite{Kussin:Meltzer:2007}, \cite{Meltzer:2007}, \cite{Dowbor:Meltzer:Mroz:2010},
\item[$\circ$] the study of infinite dimensional modules~\cite{Reiten:Ringel:2006}, \cite{Angeleri:Kussin:2013},
\item[$\circ$] the investigation of cluster categories~\cite{Barot:Kussin:Lenzing:2008},~\cite{Barot:Kussin:Lenzing:2010},
\item[$\circ$] the study of (flags of) invariant subspaces for nilpotent operators \cite{Simson:2007}, \cite{Ringel:Schmidmeier:2008}, \cite{Kussin:Lenzing:Meltzer:2012crelle}, \cite{Kussin:Lenzing:Meltzer:2012prep},
\item[$\circ$] singularity theory and categories of matrix factorizations~\cite{KST-1}, \cite{KST-2}, \cite{Lenzing:Pena:2011}, \cite{Kussin:Lenzing:Meltzer:2012adv},
\item[$\circ$] mathematical physics~\cite{Cecotti:2012},~\cite{Cecotti:DelZotto:2011}.
\end{itemize}

While for the tame domestic case, the concealed-canonical algebras are completely known through the Happel-Vossieck list~\cite{Happel:Vossieck:1983}, their structure may be quite complicated, if we allow the algebras to be tubular or wild, and many natural questions still remain open. In this paper we present another record of \emph{extremal properties}
characterizing canonical algebras, completing research by
Ringel~\cite{Ringel:2009} on the challenging question ``\emph{Why are the canonical algebras canonical?}''. We further study concealed-canonical algebras with properties antipodal to the canonical ones.

The paper is organized as follows. Section~\ref{sect:max_conditions}
presents our main results on characterizations of canonical algebras
in terms of maximality conditions. In Section~\ref{sect:setup} we
recall those properties on weighted projective lines that are
needed for their proofs
in Section~\ref{sect:proofs}. Section~\ref{sect:min_condition} deals with
concealed-canonical algebras with properties antipodal to canonical
ones. In Section~\ref{sect:instructive} we show that characterizations
of canonical algebras \emph{within the class of tame concealed
  algebras} have a tendency not to generalize to the tubular or wild
case.

As a general reference for weighted projective lines we refer
to \cite{Geigle:Lenzing:1987,Lenzing:Pena:1997}. Concerning finite
dimensional algebras and their representations, the monographs
\cite{Simson:Skowronski:2007,Simson:Skowronski:2007a} and
\cite{Ringel:1984} contain the relevant information.

\section{Extremal properties of canonical
  algebras} \label{sect:max_conditions} In this section we present the
main results of our paper, all expressing a certain extremal property
of canonical algebras.  Let $\XX=\XX(\pp,\lala)$ be a weighted
projective line given by
weight type $\pp=(p_1,\ldots,p_t)$ and parameter
sequence $\lala=(\la_3,\ldots,\la_t)$.
We denote by $\LL(\pp)$, or just $\LL$, the rank one abelian group
generated by elements $\vx_1,\ldots,\vx_t$ subject to the relations
$p_1\vx_1=\cdots=p_t\vx_t=:\vc$.
 Then $T_{\can}$, the direct sum
of all line bundles $\Oo(\vx)$ with $0\leq\vx\leq\vc$, is called the
\emph{canonical tilting bundle} on $\XX$. Its endomorphism ring is the
canonical algebra $\La=\La(\pp,\lala)$ in the sense of
Ringel~\cite{Ringel:1984}, given by the same data $\pp$ and $\lala$,
see also Section~\ref{ssect:tilting}.
Throughout we denote by $t=t(\XX)$, or just $t$, \emph{the number of
  weights} $p_i\geq2$ and by $\ovp=\ovp(\XX)$, or just $\ovp$, the
least common multiple of $p_1,\ldots,p_t$.

The complexity of the classification of indecomposables for
 $\coh\XX$, the category of coherent sheaves on $\XX$, respectively
 for the category $\mod\La$ of finite dimensional right $\La$-modules,
 is determined by the (orbifold) Euler characteristic
 of $\XX$
$$
\chi_\XX=2 - \sum_{i=1}^t \left(1-\frac{1}{p_i}\right).
$$
The representation type for both categories is tame domestic if
$\chi_\XX>0$, tame tubular for $\chi_\XX=0$ and wild for $\chi_\XX<0$,
see~\cite{Geigle:Lenzing:1987} as well as \cite{Lenzing:Reiten:2006,Lenzing:Meltzer:1993} and \cite{Lenzing:Pena:1997} concerning more specific information.

For notations and definitions, otherwise, we refer to
Section~\ref{sect:setup}.  We assume all tilting objects
$T=\oplus_{i=1}^nT_i$ on $\XX$ to be \emph{multiplicity-free}, that
is, $T_1,\ldots,T_n$ to be pairwise non-isomorphic. Throughout, we
work over an a base field $k$ which is algebraically closed. If not
stated otherwise, modules will always be right modules.

\subsection*{Maximal number of line bundles}
The following result allows two different interpretations. First, it
expresses unicity of the canonical tilting bundle if $T$ attains the
maximal possible number of line bundle summands. Second, it shows that
the same assertion holds if we minimize the differences between ranks
of the indecomposable summands of $T$. Note that the case of two
weights behaves somewhat special because there exist tilting bundles,
consisting of line bundles, whose endomorphism rings are not
canonical.

\begin{thm}[Maximal number of line bundles]
  \label{thm:main0} \label{thm:linebdle_max}
  Let $T$ be a tilting bundle on $\XX$ whose indecomposable summands
  all have the same rank $r$. Then $r$ equals one.
  Moreover, assuming $t(\XX)\neq2$, then $T$ is
  isomorphic to $T_{\can}$ up to a line bundle twist and, accordingly, the
  endomorphism ring of $T$ is isomorphic to a canonical algebra.
\end{thm}

\subsection*{Homogeneity}\label{ssect:homogeneity}

We call a tilting sheaf $T$ on $\XX$ \emph{homogeneous} in $\coh\XX$ (resp.\ in $\Der(\coh\XX)$)
if for any two indecomposable summands $T'$ and $T''$ of $T$ there exists a self-equivalence $u$ of the abelian category $\coh\XX$ (resp.\ of the triangulated category $\Der(\coh\XX)$) such that $u$  maps $T'$ to $T''$. By its very definition the canonical tilting bundle $T_\can$ is both homogeneous in $\coh\XX$ and in $\Der(\coh\XX)$. Our next theorem implies that, under mild restrictions, the canonical tilting bundle is characterized by a variety of homogeneity conditions.

\begin{thm}[Homogeneity]\label{thm:homogeneity}
Let $T$ be a tilting sheaf on $\XX$ with indecomposable summands $T_1,\ldots,T_n$. Assume that one of the following conditions is satisfied.
\begin{itemize}
\item[(i)] $T$ is homogeneous in $\coh\XX$.
\item[(ii)] $\XX$ is not tubular, and  $T$ is homogeneous in $\Der(\coh\XX)$.
\item[(iii)] $\XX$ is not tubular, and the perpendicular categories $\rperp{T_i}$, formed in $\coh\XX$ all have the same Coxeter polynomial.
\item[(iv)] $\XX$ is not tubular, and the one-point extensions $A[P_i]$ of $A=\End(T)$ with the $i$-th indecomposable projective $A$-module all have the same Coxeter polynomial.
\end{itemize}
Then all indecomposable summands of $T$ have rank one and, assuming $t(\XX)\neq2$, the tilting bundle $T$ is isomorphic to the canonical tilting bundle $T_\can$ up to
a line bundle twist.
\end{thm}
In particular, condition (iii) (resp.\ (iv)) is satisfied if the categories $\rperp{T_i}$ (resp.\ the algebras $A[P_i]$) are pairwise derived equivalent.

Assuming $\XX$ tubular, we note that Example~\ref{ex:instructiveI} presents a tilting bundle satisfying the conditions (ii), (iii) and (iv) and whose endomorphism ring is not canonical.

\subsection*{Maximal amount of bijections}
Assume $T$ is a tilting bundle on $\XX$, and $A=\End(T)$. For
$\chi_\XX\neq0$ there exists a unique generic $A$-module. If
$\chi_\XX=0$, that is, if $\XX$ and $A$ are tubular, then there exists
a rational family of generic $A$-modules, with one of them
distinguished by $T$, and called the \emph{$T$-distinguished generic}
$A$-module, a name we also use for nonzero Euler characteristic.  See
Section~\ref{ref:ssect:max_bij} for references and the relevant
definitions.

For the next result also compare~\cite{Ringel:2009}. Note that the
canonical configuration $T_{\can}$ always satisfies the condition
stated below.
\begin{thm}[Maximal amount of bijections]
  \label{thm:main:1:2} \label{thm:bij_max} Assume $t(\XX)\neq2$. Let $T$
  be a tilting bundle on $\XX$ with endomorphism ring $A$, and let $G$
  be the $T$-distinguished generic module. Then for each arrow
  $\al:u\ra v$ in the quiver of $A$ the induced $k$-linear map
  $G_\al:G_v\ra G_u$ is injective or surjective. Moreover, each $G_\al$
  is bijective if and only if
  $T=T_{\can}$, up to a line bundle twist.
\end{thm}
For $t(\XX)=2$ each tilting bundle satisfies the above bijectivity
condition. But there we will have tilting bundles with a non-canonical
endomorphism ring.

\subsection*{Maximal number of central simples}
If $T$ is a tilting bundle on $\XX$ with endomorphism ring $A$, we identify
$\mod{A}$ with the full subcategory of $\Der(\coh\XX)$, consisting of
all objects $X$ satisfying $\Hom(T,X[n])=0$ for each integer
$n\neq0$. In particular, each simple $A$-module $U$ belongs to
$\coh\XX$ or $\vect\XX[1]$ where $\vect\XX$ denotes the category of vector bundles on $\XX$. Simple $A$-modules $U$ belonging to
$\cohnull\XX$, the category of sheaves of finite length, are
called \emph{central}. Note, that this is equivalent for $U$ to have
rank zero.

\begin{thm}[Maximal number of central simples]
\label{thm:main1}\label{thm:simples_max}
Let $T$ be a tilting bundle on $\XX$ with endomorphism ring $A$. Then
the number of central simple $A$-modules is at most $n-2$, where $n$
is the rank of the Grothendieck group $\Knull(\coh\XX)$.

Moreover, the number of central simple $A$-modules equals $n-2$ if and only if $T$ equals $T_{\can}$, up to a line bundle twist.
\end{thm}
The question of the position of simple $A$-modules in the
bounded derived category
$\Der(\coh\XX)$ is also discussed in Section~\ref{ssect:min_bij}, see
further \cite{Kerner:Skowronski:2001} and
\cite[Section~5]{Lenzing:Skowronski:2003}.

\subsection*{Maximal width}
Let $T$ be a tilting bundle on $\XX$. We arrange its indecomposable
direct summands $T_i$, $i=1,\ldots,n$, such that their slopes satisfy
$\mu{T_1}\leq\mu{T_2}\leq\cdots\leq\mu{T_n}$. Then
$\width(T)=\mu{T_n}-\mu{T_1}$ is called the \emph{width of
  $T$}. Concerning slope and stability we refer to
Section~\ref{ssect:slope}.
\begin{thm}[Maximal width] \label{thm:max_width} Let $T$ be a tilting
  bundle on $\XX$ with endomorphism ring $A$. Then the width of $T$
  satisfies $\width(T)\leq\ovp(\XX)$.

  Conversely, assuming $\chi_\XX\geq0$, any tilting bundle $T$
  attaining the maximal width $\ovp(\XX)$ equals $T_{\can}$ up to a
  line bundle twist.
\end{thm}
We expect that the theorem extends to negative Euler
characteristic. In support of this, we mention the next proposition
and point to experimental evidence obtained from examples constructed
by means of H\"{u}bner reflections, compare
Section~\ref{ssect:Huebner}.
\begin{prop} \label{prop:semistable} Let $T$ be a tilting bundle on
  $\XX$ with endomorphism ring $A$. We assume that $T$ attains the  maximal
  possible width $\ovp=\ovp(\XX)$. Then the following holds:

\begin{enumerate}[\upshape (i)]
  \item Each indecomposable direct summand of $T$ of maximal (resp.\
  minimal) slope is semistable.

  \item If there exist line bundles summands $L$ and $L'$ of $T$ with
  $\mu L'-\mu L=\ovp$ such that, moreover, $L$ is a source, and
  $L'$ is a sink of the quiver of $A$, then $T=T_{\can}$ up to a line
  bundle twist.
\end{enumerate}
\end{prop}

\subsection*{Extremality of the canonical relations}
Let $R$ be a commutative, affine $k$-algebra, graded by an abelian
group $H$. If $x_1,\ldots,x_n$ are homogeneous algebra generators of
$R$, we always assume that their degrees generate $H$. We say that $R$
is a \emph{graded domain} if any product of non-zero homogeneous
elements of $R$ is non-zero. A non-zero homogeneous element $\pi$ is
called \emph{prime} is $R/R\pi$ is a graded domain. Finally, a graded
domain $R$ is called \emph{graded factorial} if each non-zero
homogeneous element of $R$ is a finite product of homogeneous
primes. Additionally, we always request that $R_0=k$ and that each
homogeneous unit belongs to $R_0$.

Our next theorem expresses a strong unicity property of the canonical
relations. Part (i) is \cite[Prop.~1.3]{Geigle:Lenzing:1987} while part (ii) is due to
S.~Mori in the $\ZZ$-graded case~\cite{Mori:1977} and to Kussin~\cite{Kussin:1998} in general.

\begin{thm}
$\mathrm{(i)}$ Let $\XX$ be a weighted projective line. Then the $\LL$-graded
coordinate algebra $S=k[x_1,\ldots,x_t]/I$, where $I$ is the ideal
generated by the canonical relations, is $\LL$-graded factorial of
Krull dimension two.

$\mathrm{(ii)}$ Assume, conversely, that $R$ is an affine $k$-algebra of
Krull-dimension two which is graded by an abelian group of rank
one. If $R$ is $H$-graded factorial then the graded algebras $(R,H)$
and $(S,\LL)$ are isomorphic, where $S=S(\pp,\lala)$ for a suitable
choice of $\pp$ and $\lala$.~\hfill$\square$
\end{thm}
We recall from~\cite{Geigle:Lenzing:1987} that the isomorphism classes of line bundles on a weighted projective line $\XX$ form a group with respect to the tensor product, called the \emph{Picard group} $\Pic\XX$ of $\XX$. By means of the correspondence $\vx\mapsto\Oo(\vx)$ we may identify $\LL$ and $\Pic\XX$. The following corollary then states an important extremality property of the canonical relations.
\begin{cor}
Let $R$ be an $H$-graded Cohen-Macaulay algebra which yields by
sheafification (Serre construction) the category $\coh\XX$ of coherent
sheaves on a weigh\-ted projective line $\XX$, and thus induces a
monomorphism of groups $j_R:H\incl \Pic\XX$, $h\mapsto
\widetilde{R(h)}$. Then $j_R$ is an isomorphism if and only if $R$ is
graded factorial, if and only if $R$ is isomorphic to an algebra
$S(\pp,\lala)$ defined by canonical relations.~\hfill$\square$
\end{cor}
As mentioned in Section~\ref{ssect:tilting} the \emph{squid} $T_\squid$ is competing with the canonical tilting bundle for the property of being the \emph{most natural tilting sheaf}.
The squid $T_{\squid}$, compared to $T_{\can}$, is accessible with
less theoretical knowledge. The squid thus is easier to construct from
general information about the category $\coh\XX$,
compare~\cite{Lenzing:1997}. On the other hand, the squid does not
contain any information on the \emph{canonical relations}
$x_i^{p_i}=x_2^{p_2}-\la_i x_1^{p_1}$, $i=3,\ldots,t$, and thus lacks information
vital for the link via the \emph{projective coordinate algebra} $S=S(\ul{p},\ul{\la})$ to other branches of mathematics, among them commutative algebra, function theory, and singularity theory.

\section{The set-up}\label{sect:setup}
We recall that we work over an algebraically closed field $k$. For the
convenience of the reader we collect relevant information about the
\emph{category of coherent sheaves} $\coh \XX$ \emph{over a weighted
  projective line} $\XX$, see \cite{Geigle:Lenzing:1987}.

\subsection{The category of coherent sheaves}
The weighted projective line is given by a weight sequence
$\ul{p}=(p_1,\ldots,p_t)$ with $p_i\geq 2$ and a parameter sequence
$\ul{\la}=(\lambda_3,\ldots,\lambda_t)$ of pairwise distinct, non-zero
elements of the field $k$. We may further assume  $\lambda_3=1$.

We recall that $\LL=\LL(\pp)$ denotes the abelian group generated by elements
$\vx_1,\ldots,\vx_t$ subject to the relations
$p_1\vx_1=p_2\vx_2=\cdots =p_t\vx_t=:\vc$. The element $\vc$ is called
the \emph{canonical element}.  The \emph{degree map} is the surjective homomorphism
defined by
\begin{equation}
\label{eq:delta}
\delta\colon \LL\rightarrow \ZZ,\quad \delta(\vx_i)=\frac{\ovp}{p_i},
\end{equation}
where $\ovp=\lcm\{p_1,\ldots,p_t\}$. The group $\LL$ has rank one with
torsion subgroup $\Ker\delta$ and is partially ordered with positive
cone $\LL_+=\sum_{i=1}^t\NN\vx_i$. This order is almost linear in the
sense that for each $\vx\in\LL$ we have the alternative
\begin{equation}\label{eq:cplusomega}
\textrm{either}\quad \vx\geq0 \quad\textrm{or}\quad \vx\leq \vc+\vom.
\end{equation}
Here, $\vom=(t-2)\vc-\sum_{i=1}^t\vx_i$ is the \emph{dualizing element} of $\LL$.

The algebra $S=k[x_1,\ldots,x_t]/I$, where $I$ is the ideal generated
by the \emph{canonical
  relations} $$x_i^{p_i}-(x_2^{p_2}-\la_ix_1^{p_1}),\quad
i=3,\ldots,t$$ is $\LL$-graded with $x_i$ being homogeneous of degree
$\vx_i$, hence
  $S=\bigoplus_{\vec{x}\in\LL_+}S_{\vec{x}}$. The group $\LL$ acts on the
  category $\mod^\LL S$ of finitely generated $\LL$-graded $S$-modules
  by \emph{degree shift} $M\mapsto M(\vec{x})$.

The category of coherent sheaves on $\XX$ is obtained
from $S$ by Serre construction (=sheafification),
compare~\cite{Serre:1955},
$$
\coh{\XX}=\mod^\LL S/\mod_0^\LL S,
$$
where $\mod_0^\LL S$ denotes the Serre
subcategory of $\mod^\LL S$ of those modules of finite length (=finite
$k$-dimension). We refer to the $\LL$-graded algebra $S$ as the
\emph{projective coordinate algebra} of $\XX$.

The action of $\LL$ on $\mod^\LL S$ induces an action on $\coh
\XX$, given by \emph{line bundle twists}, $\sigma(\vx)\colon E\mapsto
E(\vx)$ and thus determines a subgroup $\Pic(\XX)$, called the
\emph{Picard group}, of the automorphism group of $\coh \XX$.

Each coherent sheaf has a decomposition $X=X_+\oplus X_0$ where $X_0$
has finite length and $X_+$ has no simple subobject, that is, is a
\emph{vector bundle}. By $\vect\XX$ (resp.\ $\cohnull\XX$) we denote
the category of all vector bundles (resp.\ finite length sheaves).

\subsection{Serre duality}
The category $\coh \XX$, which is a connected, abelian and noetherian
category, has \emph{Serre-duality} in the form $$\dual\Ext^1_\XX(X,Y)
= \Hom_\XX(Y, X(\vom))$$ for all $X, Y \in \coh\XX$. As a consequence
$\coh\XX$ has almost-split sequences and the autoequivalence $\tau$
of $\coh\XX$, given by the line bundle twist with $\vom$, serves as the
\emph{Auslander-Reiten translation}. In particular $\tau$ preserves the rank.

\subsection{Line bundles}
By sheafification the $\LL$-graded $S$-modules $S(\vx)$ yield the
\emph{twisted structure sheaves} $\Ocr(\vx)$. Due to graded
factoriality of $S$, each line bundle $L$ in $\coh \XX$ has the form
$L=\Ocr(\vx)$ for some $\vx\in \LL$. Further, for all
$\vx,\vy\in \LL$ we obtain that
$$
\Hom_\XX(\Ocr(\vx),\Ocr(\vy))=S_{\vy-\vx}.
$$
This implies, in particular, that
$$
\Hom_\XX(\Ocr(\vx),\Ocr(\vy))\neq 0 \quad \Leftrightarrow \quad
\vx\leq\vy \textrm{ in } \LL.
$$
Invoking Serre duality we obtain the next result.
\begin{lem}
\label{lem:ext_line}
Let $L$ be a line bundle and $\vx, \vy$ be elements of $\LL$. Then we
have $\Ext^1(L(\vx),L(\vy))=0=\Ext^1(L(\vy),L(\vx))=0$ if and only if
$-\vc\leq \vy-\vx \leq \vc$.~\qed
\end{lem}

\subsection{Euler form}

The \emph{Euler form} $\eul{-,-}$ is the bilinear form on the
\emph{Grothendieck group} $\Groth(\XX)$ given on classes of objects by
$$
\eul{\class{E},\class{F}}=\dim_k \Hom(E,F)-\dim_k \Ext^1(E,F).
$$
As an abelian group $\Groth(\XX)$ is free of rank $n=2+\sum_{i=1}^t (p_i-1)$.

\subsection{Rank, degree, slope and stability}
\label{ssect:slope}

Rank and degree define linear forms
$\rk,\deg\colon\Groth(\XX)\rightarrow \ZZ$ characterized by the
properties: $\rk(\Oo(\vx))=1$ and $\deg(\Oo(\vx))=\delta(\vx)$ for
each $\vx$ in $\LL$. The rank (resp.\ degree) is strictly positive on
non-zero vector bundles (resp.\ non-zero sheaves of finite
length). Then for each non-zero sheaf $X$ the quotient $
\mu(X)={\deg(X)}/{\rk(X)} $ is a well defined member of
$\QQ\cup\{\infty\}$, called the \emph{slope} of $X$.  With these
notations, we have that
\begin{equation} \label{eq:slope+shift}
\mu(\tau E)=\mu(E)+\delta(\vom)
\end{equation}
for each object $E$.

A non-zero vector bundle $E$ is called \emph{stable} (resp.\
\emph{semistable}) if $\mu E'< \mu E$ (resp.\ $\mu E'\leq \mu E$)
holds for each proper subobject $E'$ of $E$.

\subsection{Exceptional objects and perpendicular categories}
An object $E$ in $\coh\XX$ (or more generally in its bounded derived category $\Der(\coh\XX)$)
is called \emph{exceptional} if $\End(E)=k$
and $E$ has no self-extensions which by heredity of $\coh\XX$ amounts
to $\Ext^1(E,E)=0$. Each exceptional sheaf of finite length is
concentrated in an \emph{exceptional point}, say $\lambda_i$ of weight
$p_i$, and then has length at most $p_i-1$. Each exceptional sheaf $E$
is uniquely determined by its class in $\Knull(\coh\XX)$,
see~\cite{Huebner:1996} or \cite{Meltzer:2004}.

An \emph{exceptional sequence} $E_1,\ldots,E_n$, say in $\Der(\coh\XX)$, consists of exceptional objects such that, whenever $j>i$, we have $\Hom(E_j,E_i[m])=0$ for all integers $m$. If, moreover, $n$ equals the rank of the Grothendieck group of $\coh\XX$, we call the sequence \emph{complete}. The indecomposable summands of a (multiplicity-free) tilting object $T$ in $\coh\XX$ (or $\Der(\coh\XX)$) can always arranged as a complete exceptional sequence.

If $E$ is exceptional in $\coh\XX$, its right perpendicular category $\rperp{E}$ consists of all objects $X$ from $\coh\XX$ satisfying $\Hom(E,X)=0=\Ext^1(E,X)$. It is again an abelian hereditary category with Serre duality. If $\Dd=\Der(\coh\XX)$ denotes the bounded \emph{derived} category, it is also possible to form the right perpendicular category of $E^{\perp_\Dd}$, formed in $\Dd$, consisting of all objects $X$ of $\Dd$ satisfying $\Hom(E,X[m])=0$ for each integer $m$. As is easily seen, $E^{\perp_\Dd}$ equals $\Der(\rperp{E})$.

\subsection{Coxeter polynomials} Any triangulated category $\Tt$ with a tilting object satisfies Serre duality in the form $\dual{\Hom(X,Y[1])}=\Hom(Y,\tau X)$ for a self-equivalence $\tau$ of $\Tt$. On the Grothendieck group $\Knull(\Tt)$, the equivalence $\tau$ induces an invertible $\ZZ$-linear map, the \emph{Coxeter transformation} of $\Tt$. Its characteristic polynomial is called the \emph{Coxeter polynomial} of $\Tt$. Typical instances for $\Tt$ are the (bounded) derived category $\Der{\Hh}$, where $\Hh$ is a hereditary category with Serre duality, or the (bounded) derived category $\Der(\mod A)$ of modules over a finite dimensional algebra of finite global dimension.

\subsection{Tilting objects}\label{ssect:tilting}
An object $T\in \coh \XX$ is called a \emph{tilting sheaf} if
$\Ext^1(T,T)=0$ and $T$ generates the category $\coh\XX$
homologically, in the sense that $\Hom(T,X)=0=\Ext^1(T,X)$ implies $X=0$. If further $T$ is a vector bundle, it is called a \emph{tilting bundle}.  In the terminology of \cite{Lenzing:Meltzer:1996} the endomorphism algebras of tilting bundles are \emph{concealed-canonical} and thus
by~\cite{Lenzing:Meltzer:1996} have a sincere separating tubular
family (subcategory) of stable tubes. Moreover,
by~\cite{Lenzing:Pena:1999} the existence of such a separating
subcategory characterizes concealed-canonical algebras.

 The line bundles $\Ocr(\vx)$, $ 0\leq\vx\leq\vc$, yield
the \emph{canonical tilting bundle} $T_{\can}=\bigoplus_{0\leq \vx\
  \leq \vc} \Ocr(\vx)$ for $\coh\XX$. Its endomorphism ring is given
by the quiver
\begin{center}
$$
\xymatrix@!0@R20pt@C30pt{
             && \Oo(\vx_1)\ar[rr]^{x_1}&&\Oo(2\vx_1)\ar[rr]&&\cdots\ar[rr]&&\Oo((p_1-1)\vx_1)\ar[rrddd]^{x_1}\\
             &&                  &&                 &&             &&                          \\
             &&\Oo(\vx_2)\ar[rr]_{x_2}&&\Oo(2\vx_2)\ar[rr]&&\cdots\ar[rr]&&\Oo((p_2-1)\vx_2)\ar[rrd]_{x_2}\\
\Oo\ar[rruuu]^{x_1}\ar[rru]_{x_2}\ar[rrdd]^{x_t}&&           &&                 &&     &&   &&\Oo(\vc)\\
             &&              &&     \vdots       &&             &&      \vdots\\
             &&\Oo(\vx_t)\ar[rr]^{x_t}&&\Oo(2\vx_t)\ar[rr]&&\cdots\ar[rr]&&\Oo((p_t-1)\vx_t)\ar[rruu]^{x_t}                          \\
}
$$
\end{center}
with the \emph{canonical relations}
\begin{equation} \label{eq:can_relations}
x_i^{p_i}=x_2^{p_2}-\la_i x_1^{p_1}, \quad i=3,\ldots,t.
\end{equation}
This algebra is the \emph{canonical algebra associated with} $\XX$.

Another tilting sheaf in $\coh\XX$, competing with $T_{\can}$ for the
role to be `the most natural tilting sheaf' is the \emph{squid tilting sheaf}
$T_{\squid}$, which we are going to define now.
For each $i$ from $1$
to $t=t(\XX)$ there is exactly one simple sheaf $S_i$, concentrated in
$\la_i$ satisfying $\Hom(\Oo,S_i)\neq0$. Note for this purpose that
$\la_1=\infty$ and $\la_2=0$. Moreover, there exists a sequence of
exceptional objects of finite length together with epimorphisms
$$
B_i:S_i^{[p_i-1]}\epi S_i^{[p_i-2]}\epi \cdots \epi S_i^{[1]}=S_i,
$$
where $S_i^{[j]}$ has length $j$ and top $S_i$. The direct sum of
$\Oo$, $\Oo(\vc)$ and all the $S_i^{[j]}$, $i=1,\ldots,t$,
$j=1,\ldots,p_i-1$, then forms
the tilting sheaf $T_{\squid}$
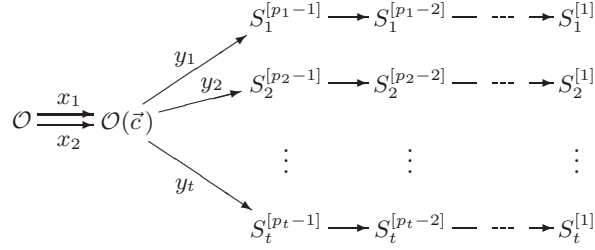
\begin{center}
\begin{picture}(225,94)
\put(45,47){%
\put(76,-1){%
	\put(0,40){\vector(1,0){15}}
	\put(0,15){\vector(1,0){15}}
	\put(0,-40){\vector(1,0){15}}%
	}%
\put(123,-1){%
\put(0,40){%
	\put(0,0){\line(1,0){9}}
	\multiput(15,0)(3,0){3}{\line(1,0){2}}
	\put(29,0){\vector(1,0){9}}
}%
\put(0,15){%
	\put(0,0){\line(1,0){9}}
	\multiput(15,0)(3,0){3}{\line(1,0){2}}
	\put(29,0){\vector(1,0){9}}
}%
\put(0,-40){%
	\put(0,0){\line(1,0){9}}
	\multiput(15,0)(3,0){3}{\line(1,0){2}}
	\put(29,0){\vector(1,0){9}}
}%
}%
\put(6,6){\vector(3,2){39}}
\put(12,3){\vector(4,1){31}}
\put(8,-8){\vector(3,-2){39}}
\multiput(-34,2)(0,-4){2}{\vector(1,0){21}}
\put(-40,0){\HVCenter{\small $\Oo$}}
\put(0,0){\HVCenter{\small $\Oo(\vc\,)$}}
\put(60,0){%
	\put(0,40){\HVCenter{\small $S_1^{[p_1-1]}$}}
	\put(0,15){\HVCenter{\small $S_2^{[p_2-1]}$}}
	\put(0,-40){\HVCenter{\small $S_t^{[p_t-1]}$}}
	}%
\put(107,0){%
	\put(0,40){\HVCenter{\small $S_1^{[p_1-2]}$}}
	\put(0,15){\HVCenter{\small $S_2^{[p_2-2]}$}}
	\put(0,-40){\HVCenter{\small $S_t^{[p_t-2]}$}}
	}%
\put(170,0){%
	\put(0,40){\HVCenter{\small $S_1^{[1]}$}}
	\put(0,15){\HVCenter{\small $S_2^{[1]}$}}
	\put(0,-40){\HVCenter{\small $S_t^{[1]}$}}
	}%
\put(60,-12.5){\HVCenter{$\vdots$}}%
\put(107,-12.5){\HVCenter{$\vdots$}}%
\put(170,-12.5){\HVCenter{$\vdots$}}%
\put(-22,8){\HVCenter{\small $x_1$}}
\put(-22,-8){\HVCenter{\small $x_2$}}
\put(22,24){\HVCenter{\small $y_1$}}
\put(32,14){\HVCenter{\small $y_2$}}
\put(22,-24){\HVCenter{\small $y_t$}}
}%
\end{picture}
\end{center}
whose endomorphism algebra is the \emph{squid algebra} $C_\squid$
associated with $\XX$. It is given by the above quiver and subject to
the relations
$$
y_1x_1=0,\; y_2x_2=0,\; y_i(x_2-\la_ix_1)=0 \textrm{ for }
i=3,\ldots,t.
$$

A less known tilting object, actually a tilting complex $T_\cox$ in
$\Der(\coh\XX)$, is displayed below. It is called \emph{Coxeter-Dynkin
  configuration of canonical type} and exists for $t(\XX)\geq2$. As
the squid it consists of two line bundles and of $t=t(\XX)$ branches
of finite length sheaves, up to translation in
$\Der(\coh\XX)$. Following \cite{Lenzing:Pena:2011}, where the dual
algebra is considered, its endomorphism ring $C_\cox$ is called a
\emph{Coxeter-Dynkin algebra of canonical type}. Such algebras,
actually their underlying bigraphs, play a prominent role in
singularity theory, compare for instance~\cite{Ebeling:2007}.
\def\nedots{\begin{picture}(20,2.0)(1,5.0) \put(2.5,2.5){$\cdot$}
    \put(5.0,5.0){$\cdot$} \put(7.5,7.5){$\cdot$}
\end{picture}}

\def\needots{\mbox{\begin{picture}(12,1.0)(1,5.0)
\put(2.2,3.75){$\cdot$}
\put(5.0,5.0){$\cdot$}
\put(7.8,6.25){$\cdot$}
\end{picture}}}

\def\seedots{\begin{picture}(12,1.0)(1,5.0)
\put(2.2,6.25){$\cdot$}
\put(5.0,5.0){$\cdot$}
\put(7.8,3.75){$\cdot$}
\end{picture}}
\begin{figure}[H]
$$
\xymatrix@!0@C32pt@R32pt{
&&S_1^{(p_1-1)}\ar[rrdd]^<<<<<{\al_1}\ar[rr]&&S_1^{(p_1-2)}\ar[rr]&&\cdots\ar[r]&S_1^{(2)}\ar[rr]&&S_1^{(1)}&\\
&&S_2^{p_2-1}\ar[rrd]_{\al_2}\ar[rr]&&S_2^{(p_2-2)}\ar[rr]&&\cdots\ar[r]&S_2^{(2)}\ar[rr]&&S_2^{(1)}&\\
\Oo(\vc)\ar[rruu]^<<<<<<<<<<<<<<<{\al_1}\ar[rru]_{\al_2}\ar[rrd]^{\al_t}&&\vdots&&\Oo(-\vom)[1]&&&&\vdots&\\
&&S_t^{(p_t-1)}\ar[rru]^{\al_t}\ar[rr]&&S_t^{(p_t-2)}\ar[rr]&&\cdots\ar[r]&S_t^{(2)}\ar[rr]&&S_t^{(1)}&\\
}
$$
\caption{Coxeter-Dynkin algebra of canonical type}
\end{figure}
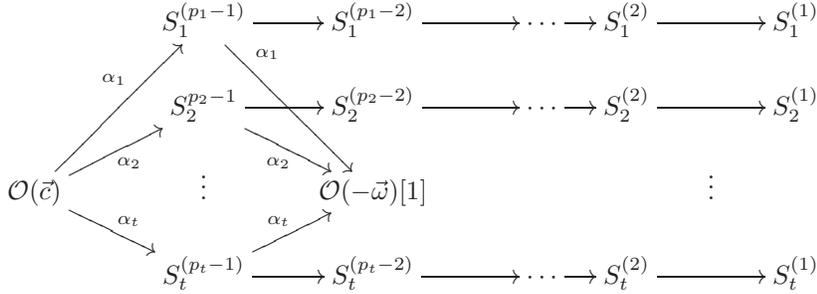
The endomorphism algebra $C_\cox$ of $T_\cox$ is given by the above `quiver' with the two relations
\begin{equation}\label{eq:cox_dynkin}
\sum_{i=2}^t
\al_i^2=0 \quad \text{and} \quad \al_1^2=\sum_{i=3}^t \la_i\,\al_i^2.
\end{equation}
It is remarkable that $C_\cox$ is Schurian for $t(\XX)=3$. Moreover,
for $t(\XX)\geq 5$ the number of relations is strictly less than for
the canonical algebra or the squid algebra. In the tubular case, and
only there, $C_\cox$ can be realized as the endomorphism ring of a
tilting sheaf, actually as the endomorphism ring of a tilting bundle,
see Section~\ref{ssect:tubular_width} for an interesting extremal
property of these algebras.

\subsection{Tubular mutations}\label{ssect:tub_mut}
Assume $\XX$ has Euler characteristic
zero. Tubular mutations are distinguished self-equivalences of
$\Der(\coh\XX)$ playing a key role in the classification of indecomposable
objects. By tilting they are related to Ringel's shrinking
functors from~\cite{Ringel:1984}. Their formal definition is due
to~\cite{Lenzing:Meltzer:1993}. From different perspectives, the
subject is also treated in \cite{Meltzer:1997},
\cite{Lenzing:Pena:1999}, \cite{Kussin:2009}, \cite{Lenzing:2007}. For
quick information we recommend the survey in \cite{Meltzer:2004}.

The tubular mutation $\rho:\Der{\coh\XX}\ra \Der{\coh\XX}$ is a triangle equivalence that
is given on indecomposable objects of slope $\mu{X}>0$ by the universal extension
\begin{equation}
0\lra \bigoplus_{j=1}^{\bp} \Ext^1(X,\tau^j\Oo)\otimes_k \tau^j\Oo \lra \rho X \lra X \lra 0,
\end{equation}
see \cite{Lenzing:Meltzer:1993} or \cite[Section~10.3]{Lenzing:2007}. Another self-equivalence
of $\Der{\coh\XX}$, actually also a tubular mutation, is given by the line bundle shift $\si(X)=X(\vx_t)$, where $\vx_t$ belongs
to the largest weight of $\XX$, thus $\delta(\vx_t)=1$. On the pair $(d,r)=(\deg X, \rk X)$,
the actions induced by $\si$ and $\rho$ are given by the right multiplication with the
matrix $\left(\begin{array}{cc}1&0\\1&1\end{array}\right)$ resp.\ $\left(\begin{array}{cc}
1&1\\0&1
\end{array}\right)$. In particular, $\si$ (resp.\ $\rho$) preserves the rank (resp.\ the degree).
Further, $\si$ (resp. $\rho$) induces an action on slopes, given by the fractional linear
transformation $q\mapsto 1+q$ (resp.\ $q\mapsto q/(1+q)$).

The self-equivalences $\si$ and $\la=\rho^{-1}$ are conjugate, actually
\begin{equation}
\label{eq:conj_from_braid}
\la=(\la\si)^{-1}\si(\la\si)
\end{equation}
which follows from the braid
relation $\si\la\si=\la\si\la$,
see for instance \cite[prop.\ 6.2]{Lenzing:Meltzer:2000}.

\subsection{H\"{u}bner reflections}\label{ssect:Huebner}
A useful tool to construct new tilting sheaves from given ones is by
means of mutations (more precisely, by \emph{H\"{u}bner
  reflections}). We recall the relevant facts from
~\cite{Huebner:1996}, see also~\cite{Huebner:1997}. Let $T$ be a
(multiplicity-free) tilting sheaf on a weighted projective line $\XX$
and assume that $T=T'\oplus E$ with $E$ indecomposable. Then there
exists exactly one indecomposable object $E^\ast$ not isomorphic to
$E$ such that $T^\ast=T'\oplus E^\ast$ is again a tilting sheaf. We
say that $T^\ast$ is the mutation of $T$ at the vertex (corresponding
to) $E$. In more detail, let $Q$ be the quiver of the endomorphism
algebra $A$ of $T$. Let $T_1,\ldots,T_n$ be the (non-isomorphic)
indecomposable summands of $T$ and let $U_1,\ldots,U_n$ be the
corresponding simple $A$-modules, viewed as members of $\Der(\coh\XX)$
where we identify $T_i$ with the $i$-th indecomposable projective
$A$-module. A vertex is called a \emph{formal source} (resp.\ a
\emph{formal sink}) if $U_i$ belongs to $\coh\XX$ (to $(\coh\XX)[1]$,
respectively). Each source (resp.\ sink) of $Q$ is a formal source
(resp.\ a formal sink). Moreover, each vertex $i$ of $Q$ is either a
formal source or a formal sink. Assume that $i$ is a formal sink. Then
there exists an exact sequence
\begin{equation}\label{eq:Huebner:reflection}
0\lra T_i^* \stackrel{u}{\lra} \bigoplus_{j=1}^n T_j^{\kappa_j} \lra T_i \lra 0
\end{equation}
where $\kappa_j$ denote the number of arrows from $j$ to $i$ in
$Q$, and where $u$ collects these arrows. Moreover, we have $\Ext^1(T_i,T_i^\ast)=k$ and
$\Ext^1(T_i^\ast,T_i)=0$. The case of a formal source is dual.

\section{Proofs and more}\label{sect:proofs}
We fix a tilting bundle $T$ on $\XX$ and denote by $A=\End(T)$ its
endomorphism ring. We recall the previous convention to consider
$\mod{A}$ as a full subcategory of $\coh\XX \vee \coh\XX[1]$, and use
the notation $S_1,\ldots,S_n$ for the corresponding simple
$A$-modules. Let $\bw$ denote the class of any homogeneous simple
sheaf $S_0$.  Note that $\rk(x)=\gen{x,\bw}=-\gen{\bw,x}$ for
  each $x\in\Groth(\XX)$.

\begin{lem}[{\cite[Prop.\ 4.26]{Huebner:1996}}] \label{eq:Huebner(a)} \label{eq:Huebner(b)}
Assume $T$ is a tilting bundle on $\XX$. With the preceding notations we have
\begin{equation*}
\sum_{i=1}^n \rk (T_i) \class{S_i}=\bw, \quad
\sum_{i=1}^n \rk (S_i) \class{T_i}=-\bw.
\end{equation*}
\end{lem}
\begin{proof}
  Indeed, the Euler form $\gen{-,-}$ satisfies
  $\gen{\class{T_i},\class{S_j}}=\delta_{ij}$.  Expressing $\bw$ in
  the basis of the simples $\bw=\sum_{i=1}^n\al_i \class{S_i}$, we
  get
$$
\rk(T_j)=\gen{\class{T_j},\bw}=\sum_{i=1}^n\al_i\gen{\class{T_j},\class{S_i}}=\sum_{i=1}^n
\al_i \delta_{ij}=\al_j,
$$
and the first formula follows. Expressing $-\bw$ in the base of the
projectives $-\bw=\sum_{i=1}^n \beta_i \class{T_i}$, we get
$$
\rk(S_j)=\gen{\class{S_j},\bw}=-\gen{\bw,\class{S_j}}=\sum_{i=1}^n \beta_i \gen{\class{T_i},\class{S_j}}=\beta_j.
$$
This shows the
second formula.
\end{proof}

\subsection{Maximal number of central simples} We are now going to
prove Theorem \ref{thm:main1}. Denote by $S_1,\ldots,S_n$ the simple
$\End(T)$-modules corresponding to the projectives $T_1,\ldots,T_n$
respectively. We may assume that $T_1,\ldots,T_n$ form an exceptional sequence, a fact implying that the simples $S_n,\ldots,S_1$ form an exceptional
sequence in the reverse direction.

To prove the first claim of the theorem, we observe that the vertex
associated to $T_1$, resp.\ to $T_n$ is a source, resp.\ a sink, of the
quiver of $A=\End(T)$. Hence $S_1=T_1$ is simple projective over $A$ of
positive rank and $S_n=\tau T_n[1]$ is simple injective over $A$ of
negative rank. Since central simple $A$-modules have rank zero, we
conclude that the number of central simple modules is at most
$n-2$. The bound $n-2$ is actually attained for the canonical tilting
bundle $T_\can$: Then $0\ra \Oo((j-1) \vx_i)\ra \Oo(j\vx_i)\ra
S_{i,j}\ra 0$ is a projective resolution of the simple $A$-module
$S_{i,j}$ associated to the projective $\Oo(j\vx_i)$ for
$j=1,\ldots,p_i-1$. Hence $\rk S_{i,j}=\rk \Oo(j\vx_i)-\rk
\Oo((j-1)\vx_i)=1-1=0$ showing that $S_{i,j}$ is a central simple
$A$-module.

We next assume that $T$ is a tilting bundle with $n-2$ central simple
modules over $A$. As before we conclude that $S_1=T_1$ (resp.\ $S_n=\tau T_n[1]$) is a
simple projective (resp.\ simple injective) $A$-module of positive
(resp.\ negative) rank. Hence the simple $A$-modules
$S_2,\ldots,S_{n-1}$ have rank zero.

Applying Lemma~\ref{eq:Huebner(b)} we obtain $\rk
(S_1)\class{T_1}+\rk(S_n)\class{T_n}=-\bw$. Since $S_1=T_1$ and
$S_n=\tau T_n[1]$ we have $\rk(S_1)=\rk(T_1)$ and
$\rk(S_n)=-\rk(T_n)$.  Consequently
\begin{equation}
\label{eq:p1pn}
\rk(T_1)\class{T_1}+\bw=\rk(T_n)\class{T_n}
\end{equation}
Applying the rank function to \eqref{eq:p1pn} we get $(\rk T_1)^2=(\rk
T_n)^2$, and conclude that $\rk T_1=\rk T_n$ since both values
are positive. Calling this common value $\rho$, then \eqref{eq:p1pn} implies that
$\bw=\rho(\class{T_n}-\class{T_1})$. Since $\bw$ is indivisible in the
Grothendieck group $\Knull(\coh\XX)$ we further get $\rho=1$. Hence
$T_1=L$ and $T_n$ are line bundles and now \eqref{eq:p1pn}
implies that $T_n=L(\vc)$.

If $\deg(S_1)=q$, it follows from \eqref{eq:slope+shift} that
  $\deg(S_n)=\deg(\tau T_n[1])=-\deg(\tau
  T_n)=-(\deg(L(\vc))+\delta(\vom))=-(q+\ovp+\delta(\vom))$.  Invoking $\deg \bw=\ovp$
and additionally Lemma~\ref{eq:Huebner(a)} we obtain
\begin{align*}
\ovp&=\rk T_1 \deg S_1 +\rk T_n \deg S_n +\sum_{h=2}^{n-1} \rk T_h \deg{S_h}\\
&=-(\ovp+\delta(\vom))+\sum_{h=2}^{n-1} \rk T_h \deg{S_h}.
\end{align*}
Hence we get
\begin{equation}
\label{eq:sandwich1}
2\ovp+\delta(\vom)=\sum_{h=2}^{n-1} \rk T_h \deg{S_h}
\geq \sum_{h=2}^{n-1}  \deg{S_h}.
\end{equation}
By assumption the $A$-modules $S_i$, $i=2,\ldots,n-1$, are simple
exceptional sheaves of rank zero. Since further $S_{n-1}, \ldots, S_2$
form an exceptional sequence, each exceptional tube of $\cohnull\XX$
with $p_j$ simple sheaves can contain at most $p_j-1$ simple
$A$-modules. Since $\sum_{j=1}^t(p_j-1)=n-2$, our assumption on the
number of central simples implies that each exceptional tube of rank
$p_j$ contains exactly $p_j-1$ of them. Using further that $\deg
(S_h)\geq\frac{p}{p_j}$ if $S_h$, $h=2,\ldots,n-1$, belongs to an
exceptional tube of rank $p_j$, we thus obtain:
\begin{equation}
\label{eq:sandwich2}
\sum_{h=2}^{n-1}  \deg{S_h}\geq \sum_{j=1}^{t}
(p_j-1)\frac{\ovp}{p_j}=t\cdot\ovp+\sum_{j=1}^t\frac{1}{p_j}=2\ovp+\delta(\vom).
\end{equation}
This implies that inequality \eqref{eq:sandwich1} is indeed an
equality, which proves that $\rk T_i=1$ for all $i=1,\ldots,n$.

Thus $T$ is a direct sum of line bundles $T_i=L(\vy_i)$ and, moreover,
$T_1=L$ and $T_n=L(\vc)$. Applying Lemma \ref{lem:ext_line} to the
pairs $L$, $L(\vy_i)$ and $L(\vy_i)$, $L(\vc)$, we then obtain $0\leq
\vy_i\leq \vc$.  This shows that $T=\bigoplus_{0\leq \vx\leq \vc}
L(\vx)$ is the canonical tilting bundle up to a line bundle twist,
and thus finishes the proof of Theorem \ref{thm:main1}.


\subsection{Maximal number of line bundles}
We now prove Theorem \ref{thm:main0}.  Since $T=\bigoplus_{i=1}^n T_i$
is tilting, we obtain $ [\Oo]=\sum_{i=1}^n m_i[T_i]$  with
$m_i\in\ZZ$. Passing to ranks we obtain that the common rank $r$ of
the $T_i$ divides $1$. Hence $r=1$ follows, and thus
$T=\bigoplus_{\vy\in J} \Ocr(\vy)$ for some subset $J\subset \LL$ of
cardinality $n=2+\sum_{i=1}^t(p_i-1)$ where $\pp=(p_1,\ldots,p_t)$ is
the weight sequence of $\XX$. By means of a line bundle twist, we
may assume that (i) $0\in J$ and, moreover, (ii) $0\leq\delta(\vx)$
for all $\vx\in J$.

Lemma \ref{lem:ext_line} implies (iii) $-\vc \leq \vx\leq \vc$ for
each $\vx\in J$. Invoking the normal form $\vx=\sum_{i=1}^t
\ell_i\vx_i+\ell\vc$ with $0\leq \ell_i<p_i$ and $\ell\in\ZZ$,
conditions (ii) and (iii) imply that $\ell\in\{-1,0,1\}$. Note that
$\ell\in\{0,1\}$ implies that $\vx=a\vx_i$ for some $i=1,\ldots,t$ and
$0\leq a \leq p_i$. If $\ell=-1$, then the inequality $0\leq \vx+\vc
=\sum_{i=1}^t \ell_i\vx_i\leq 2\vc$ shows that exactly two of the
summands $\ell_i\vx_i$ are non-zero. Hence $\vx+\vc=\ell_i\vx_i
+\ell_j\vx_j$ with $i\neq j$ and then $\vx=a_i\vx_i-b_j\vx_j$ with
$0<a_i<p_i$ and $0<b_j<p_j$. In the first case, where $\ell\in\{0,1\}$, we call $\vx$
\emph{unmixed}, in the second case, where $\ell=-1$, $\vx$ is called \emph{mixed}.

We distinguish the two cases (a) $\vc\in J$ and (b) $\vc\notin J$.

\emph{Case (a):} If $\vc$ belongs to $J$, then
Lemma~\ref{lem:ext_line} implies $0\leq\vx\leq\vc$ for each $\vx\in J$
implying that $T=T_\can$ by cardinality reasons.

\emph{Case (b):} If $\vc$ does not belong to $J$, then $J$ contains a
mixed element, say, $\vy=a_1\vx_1-a_2\vx_2$ with $0<a_1<p_1$ and
$0<a_2<p_2$.

We are going to show that then $t(\XX)=2$ and first claim that $J\subset
\ZZ\vx_1+\ZZ\vx_2$. Indeed let $0\neq\vx\in J$, say,
$\vx=b_i\vx_i-b_j\vx_j$ where $i\neq j$, $i,j=1,\ldots,t$,
$0<b_i<p_i$, and $0\leq b_j<p_j$. Note that $b_i=0$ is impossible
since $\vx\neq0$ and $\delta(\vx)\geq0$. By Lemma~\ref{lem:ext_line}
we get $0\leq \vy-\vx +\vc \leq 2\vc$, hence
$$
0\leq a_1\vx_1 +b_j\vx_j +(p_2-a_2)\vx_2-b_i\vx_i \leq 2\vc.
$$
Since $i\neq j$ this is only possible if $i\in\{1,2\}$. If $b_j=0$,
then $\vx=b_i\vx_i$ belongs to $\ZZ\vx_1+\ZZ\vx_2$. If $b_j\neq0$,
then reversing the roles of $\vx$ and $\vy$ in the preceding argument
we get that also $j$ belongs to $\{1,2\}$. Summarizing we obtain that
$J\subset \ZZ\vx_1+\ZZ\vx_2$.

Finally assume that $t(\XX)\geq3$. Let $S_3$ be a simple sheaf
concentrated in the third exceptional point of weight $p_3$ and such
that $\Hom(\Oo,S_3)=0$. Since $S_3(\vx_i)=S_3$ for each $i\neq 3$ we
obtain that $\Hom(\Oo(\vx),S_3)=0$ for each $\vx$ from
$\ZZ\vx_1+\ZZ\vx_2$. In particular, we get $\Hom(T,S_3)=0$. Because
$T$ is a vector bundle, we further obtain $\Ext^1(T,S_3)=0$,
contradicting the assumption that $T$ is tilting.  This finishes the
proof of Theorem~\ref{thm:linebdle_max}.

\subsection{Homogeneity}\label{ref:ssect:homogeneity}
Let $A$ be a finite dimensional $k$-algebra, and $E$ a finite dimensional right
$A$-module. Then the $k$-algebra
$$
A[E]=\left(\begin{array}{cc}A&0\\ E&k\end{array}\right)
$$
is called the \emph{one-point extension} of $A$ by $E$. The proof of
Theorem~\ref{thm:homogeneity} is based on the next proposition.

\begin{prop} \label{prop:homogeneous}
Let $T$ be a tilting sheaf in $\coh\XX$ with indecomposable summands $T_1,\ldots,T_n$. Consider the following properties:

$(a)$ $T$ is homogeneous in the abelian category $\coh\XX$.

$(b)$ $T$ is homogeneous in the triangulated category $\Der(\coh\XX)$.

$(c)$ The perpendicular categories $\rperp{T_i}$, formed in $\coh\XX$, are pairwise derived equivalent.

$(d)$ The Coxeter polynomial $\psi'_i$ of $\rperp{T_i}$ does not depend on $i=1,\ldots,n$.

$(e)$ The Coxeter polynomial $\bar\psi_i$ of the one-point extension $A[P_i]$ of $A=\End(T_i)$ by the indecomposable projective $A$-module $P_i$, corresponding to $T_i$, does not depend on $i=1,\ldots,n$.

Then we always have the implications $(a)\implies (b) \implies (c) \implies (d) \iff (e)$. Moreover, if $\XX$ is not tubular, the indecomposable summands of $T$ are line bundles, forcing the equivalence of $(a)$ to $(e)$.
\end{prop}

\begin{proof}
$(a)\implies (b)$: Each self-equivalence of $\coh\XX$ extends to a self-equivalence of $\Der(\coh\XX)$.

$(b)\implies (c)$: With $\Dd=\Der(\coh\XX)$, the right perpendicular category $T_i^{\perp_\Dd}$ formed in $\Dd$, equals the derived category of $\rperp{T_i}$.

$(c)\implies (d)$: The Coxeter polynomial is preserved under derived equivalence.

$(d)\iff (e)$: By \cite[Prop. 18.3 and Cor. 18.2]{Lenzing:1999}, see also \cite[Prop. 4.5]{Lenzing:Pena:2008}, the Coxeter polynomials $\psi'_i$ of $\rperp{T_i}$ and $\bar{\psi}_i$ of the one-point extension $A[P_i]$ are related by the reciprocity formula
\begin{equation} \label{eq:reciprocity}
P_{T_i}=\frac{\psi-x\psi'_i}{\psi}=\frac{\bar{\psi_i}-x\psi}{\psi},
\end{equation}
where $\psi$ denotes the Coxeter polynomial of $\coh\XX$, and
\begin{equation} \label{eq:Hilbert:Poincare}
P_{T_i}=\sum_{n=0}^\infty \LF{\class{T_i}}{\class{\tau^n T_i}}\,x^n
\end{equation}
denotes the \emph{Hilbert-Poincar\'{e} series} of $T_i$. The equivalence of conditions $(d)$ and $(e)$ is now implied by formula \eqref{eq:reciprocity}, thus finishing the proof of the first claim.

We next assume that $\XX$ is not tubular, and that $(d)$ or $(e)$ holds which forces by \eqref{eq:reciprocity} all the $P_{T_i}$ to be equal.  We claim that all the $T_i$ have the same rank. Let $\al_m$ denote the $m$-th coefficient of $P_{T_i}$. Denoting, as usual, by $\bp$ the least common multiple of the weights, we have for each $x$ in $\Knull(\coh\XX)$
$$
\tau^{\bp}x= x + \rk(x)\,\delta(\vom)\,\bw,
$$
where $\bw$ denotes the class of any ordinary simple sheaf on $\XX$. This is implied by the formula $\bp\vom=\delta(\vom)\vc$ from~\cite{Geigle:Lenzing:1987}. It follows that
$$
\al_{\bp}-\al_0=\rk(T_i)\,\delta(\vom)\LF{\class{T_i}}{\bw}=\rk(T_i)^2\,\delta(\vom),
$$
and thus $\rk(T_i)^2\,\delta(\vom)$ does not depend on $i$.
By our assumption $\delta(\vom)\neq0$, then  the rank of $T_i$ does not depend on $i=1,\ldots,n$. Therefore by Theorem~\ref{thm:main0} all the $T$ are line bundles, forcing $T$ to be homogeneous.
\end{proof}
\begin{proof}[Proof of Theorem~\ref{thm:homogeneity}]
By \cite{Lenzing:Meltzer:2000} each self-equivalence of $\coh\XX$ is rank preserving. Hence
the indecomposable summands of each homogeneous tilting sheaf have the same rank. By Theorem~\ref{thm:main0} thus property (i) implies the claim. For the remaining properties (ii) to
(iv) the claim follows from Proposition~\ref{prop:homogeneous}.
\end{proof}

\subsection{Maximal amount of bijections} \label{ref:ssect:max_bij}
We recall that a module $G$ over a finite dimensional $k$-algebra $A$
is called \emph{generic} if (i) $G$ is indecomposable, (ii) $G$ has
finite length over $\End(G)$, and (iii) $G$ has infinite
$k$-dimension. Now let $A$ be an \emph{almost concealed-canonical}
algebra, that is, the endomorphism algebra of a tilting sheaf $T$ on a
weighted projective line $\XX$. The injective hull $\EE(\Oo)$ of the
structure sheaf in the category $\Qcoh\XX$ of quasicoherent sheaves on
$\XX$ equals the \emph{sheaf of rational functions} $\mathcal{K}$.
Under the equivalence $\mathrm{R}\Hom(T,-):\Der(\Qcoh\XX)\ra \Der(\Mod
A)$ the sheaf $\mathcal{K}$ corresponds to a generic $A$-module $G$,
called the \emph{$T$-distinguished generic $A$-module}. It is known,
compare~\cite{Lenzing:1997a,Reiten:Ringel:2006}, that $G$ is the
unique generic $A$-module if $\chi_\XX\neq0$. In the tubular case
$\chi_\XX=0$ the situation is more complicated, since there exists a
rational family $(G^{(q)})_q$ of generic $A$-modules, indexed by a set
of rational numbers; for details we refer to \cite{Lenzing:1997a}.  In
the proper formulation our next result extends also to the generic
modules $G^{(q)}$. Details are left to the reader. Here, we restrict to
the $T$-distinguished case.

We now prove Theorem~\ref{thm:bij_max}. We view $G=\mathcal{K}$ as a
(contravariant) representation of the quiver of $A$. First we show
that each arrow $\al: u \ra v$ induces a monomorphism or an
epimorphism $G_\al:G_v\ra G_u$. It follows from \cite{Lenzing:1997a}
that the endomorphism ring of $G$ equals the rational function field
$K=k(x)$ and, moreover, $\rk T_u=\dim_K(G_u)$ for each vertex $u$ of
the quiver of $A$. By a result of Happel-Ringel~\cite[Lemma
4.1]{Happel:Ringel:1982}, the map $T_\al:T_u\ra T_v$ is a monomorphism
or an epimorphism since $\Ext^1(T_v,T_u)=0$. Since $G$ is injective in
the category of quasi-coherent sheaves on $\XX$ this yields that
$G_\al=\Hom(\al,G):G_v\ra G_u$ is an epimorphism or a monomorphism of
$K$-vector spaces. Hence $G_\al$ is bijective if and only if $\dim_K
G_v=\dim_K G_u$, that is, if and only if $\rk T_u=\rk T_v$. In
particular, for the canonical tilting bundle all $G_\al$ are
bijective. Conversely, assuming that all $G_\al$ are bijective,
connectedness of the quiver of $A$ implies that all $T_u$,
$u=1,\ldots,n$, have the same rank, and then
Theorem~\ref{thm:linebdle_max} implies that $T$ equals $T_\can$ up to
a line bundle twist.

\subsection{Maximal width}
The proof of Theorem~\ref{thm:max_width} is based on the following proposition.
\begin{prop}\label{prop:hom+ext}
  Let $E$ and $F$ be non-zero vector bundles on a weighted projective
  line $\XX$ of arbitrary weight type. Then the following properties
  hold:
\begin{enumerate}[\upshape (i)]
\item If $\mu F -\mu E > \delta(\vc+\vom)$ then $\Hom(E,F)$ is non-zero.
\item If $\Ext^1(F,E)=0$ then $\mu F-\mu E\leq \delta(\vc)$.
\end{enumerate}
\end{prop}
\begin{proof}
  Property (i) is shown in {\cite[Thm.\ 2.7]{Lenzing:Pena:1997}}. Further, property (ii) follows from (i)
  by Serre duality.
\end{proof}

It amounts to a significant restriction for $E$ and $F$ to attain the
bound for the slope in part (ii).

\begin{cor}\label{cor:semistable}
  Let $E$ and $F$ be non-zero vector bundles with slope difference $\mu F-\mu
  E=\ovp$ and satisfying
  $\Ext^1(F,E)=0$. Then $E$ and $F$ are semistable.
\end{cor}
In particular, if $E$ and $F$ are indecomposable and $\chi_\XX<0$ then $E$ and $F$ are quasi-simple in their respective Auslander-Reiten components which have type $\ZZ\AA_\infty$.
\begin{proof}
  By symmetry it suffices to show that each non-zero subobject $F'$ of
  $F$ has a slope $\mu F'\leq \mu F$. Indeed, since the category
  $\coh\XX$ is hereditary, vanishing of $\Ext^1(F,E)$ implies that
  $\Ext^1(F',E)=0$. By Proposition~\ref{prop:hom+ext} thus $\mu F'-\mu
  E\leq \ovp$. This forces $\mu F'\leq \mu F$ and proves the
  semistability of $F$.
\end{proof}

\begin{proof}[Proof of Theorem~\ref{thm:max_width}]
  Let $T$ be a tilting bundle on $\XX$ with $\mu T_1\leq \ldots\leq
  \mu T_n$.  Since $0=\Ext^1(T_n,T_1)=\dual\Hom(T_1,T_n(\vom))$ we get
  by the preceding proposition that $\mu T_n(\vom)-\mu
  T_1\leq\ovp+\delta(\vom)$.  Since $\mu T_n(\vom)=\mu
  T_n+\delta(\vom)$ we conclude that $\mu T_n-\mu T_1\leq \ovp$,
  showing that the width $\width(T)$ is bounded by $\ovp$.

The bound $\ovp$ for the width is clearly attained for the canonical
tilting bundle since $0=\mu \Oo$ and $\mu\Oo(\vc)=\ovp$.

We now assume that $\chi_\XX\geq 0$ and that $T$ is a tilting bundle
on $\XX$ with width $\width(T)=\ovp$.
We set $F=T_n(-\vc)$ and observe
that $\mu F=\mu T_1$ and $[T_n]=[F]+\rk (F)\bw$ in $\Knull(\coh \XX)$.
Since $\Hom(T_n,T_1)=0$ by semistability and $\Ext^1(T_n,T_1)=0$ we
get $0=\gen{[T_n],[T_1]}=-\gen{[T_1],[T_n(\vom)]}$ and therefore
\begin{align}
\tag*{}
0=\gen{[T_1],[T_n(\vom)]}=&\gen{[T_1],[F(\vom)]}+\rk F\gen{[T_1],\bw}\\
\tag*{}
&=\gen{[T_1],[F(\vom)]}+\rk T_1\rk F\\
\label{eq:f_t_1}
&=-\gen{[F],[T_1]}+\rk T_1\rk T_n.
\end{align}
Since $T_1$ and $T_n$ have positive rank this implies that
$\Hom(F,T_1)\neq 0$. We will show that $F= T_1$ is a line bundle.  If
$\chi_\XX=0$, then $F$ and $T_1$ must lie in the same tube.
By exceptionality, they further have a quasi-length less than the rank of
the tube such that, in particular, $\dim \Hom(F,T_1)\leq
1$. Then~\eqref{eq:f_t_1} implies that $\Ext^1(F,T_1)=0$ and
$\dim\Hom(F,T_1)=1$,
thus $\rk T_1=\rk F=1$ and finally $F=T_1$.
If $\chi_\XX>0$, then $F=T_1$ follows  by stability since $F$ and $T_1$ have the
same slope and $\rk T_1=1$ follows again by~\eqref{eq:f_t_1}.

The argument also shows that $T_n$ is the unique indecomposable
summand of $T$ of maximal slope $\mu(T_n)$, and dually $T_1$ is the
unique indecomposable summand of $T$ having minimal slope.  By
semistability this implies $\Hom(T_i,T_1)=0=\Hom(T_n,T_i)$ for all
$1<i<n$ showing that $T_1$ is a source and $T_n$ a sink of the quiver
of $\End(T)$. We finally get that $T$ is the canonical tilting bundle,
up to a line bundle twist, by applying
Proposition~\ref{prop:semistable} whose proof is given below. This
will conclude the proof of Theorem~\ref{thm:max_width}.
\end{proof}

\begin{proof}[Proof of Proposition~\ref{prop:semistable}]
Assertion (i) is a special case of Corollary~\ref{cor:semistable}.

Concerning assertion (ii) let $L$ and $L'$ be line bundle summands of
$T$, corresponding to a sink (resp.\ a source) of the quiver of $A$
and satisfying the maximality property $\mu(L')-\mu(L)=\ovp$. Since
$L'$ and $L(\vc)$ have the same degree, we notice first that
$L'=L(\vc+\vx)$ for some $\vx$ of degree zero. Because
$0=\Ext^1(L',L)=\mathrm{D}\Hom(L,L'(\vom))$ we obtain
$\vc+\vom+\vx\leq \vc+ \vom$, hence $\vx\leq0$. Since $0\geq\vx$ and
$\vx$ has degree zero, we obtain $\vx=0$, implying that $L'=L(\vc)$.
Because of the maximality property $\mu(L(\vc))=\mu(L)+\ovp$, each
direct summand $T_i$ of $T$ satisfies $\mu(L)\leq \mu(T_i)\leq \mu
L(\vc)$ by Proposition~\ref{prop:hom+ext}.

By our assumption for $L$ (resp.\ $L(\vc)$) to correspond to a source
(resp.\ a sink) of $A$, we may assume that $L=T_1$ and $L(\vc)=T_n$.
Thus as in the proof of Theorem~\ref{thm:simples_max}
we have that
$S_1=T_1$ and $S_n=\tau T_n[1]=T_1(\vc+\vom)[1]$ and
$\rk(S_n)=-\rk(T_n)=-1$, where $S_1,\ldots,S_n$ denote the simple
$A$-modules corresponding to the indecomposable projective $A$-modules
$T_1,\ldots,T_n$.  Invoking Lemma~\ref{eq:Huebner(b)} we obtain
$$
\class{T_1}-\class{T_n} + \sum_{i=2}^{n-1}\rk S_i \class{T_i}=-\bw.
$$
Since $T_1$ is a line bundle, we have equality
$\class{T_1(\vc)}=\class{T_1}+\bw$ implying
$$
\sum_{i=2}^{n-1}\rk S_i \class{T_i}=0.
$$
Since the classes $\class{T_1},\ldots,\class{T_n}$ are linear
independent in $\Knull(\coh\XX)$, each simple $A$-module $S_i$ with
$i=2,\ldots,n-1$ has rank zero. Hence $A$ has the maximal possible
number of central simple modules. By Theorem~\ref{thm:simples_max} we
then conclude that $T=T_\can$ up to a line bundle twist.
\end{proof}

\subsection{An addendum: tubular width}\label{ssect:tubular_width}
For non-zero Euler characteristic the ``distance'' $|\mu Y-\mu X|$ of
a pair of objects is an invariant with respect to the autoequivalences
of $\Der(\coh\XX)$. This is no longer true for Euler characteristic
zero where the ``tubular distance'' given by the absolut value
of $$\rk X\rk Y(\mu Y-\mu X)=
\begin{vmatrix}
  \rk X & \rk Y\\
  \deg X&\deg Y
\end{vmatrix}=\DLF{X}{Y}$$ serves as a proper
replacement. Here
$$\DLF{X}{Y}=\sum_{j\in\ZZ_{\ovp}}\LF{X}{\tau^jY},\quad\ovp=\lcm(p_1,\dots,p_t)$$
is an average of the Euler form, and the above equality is
Riemann-Roch's theorem for a tubular weighted projective line $\XX$,
see~\cite{Lenzing:Meltzer:1993}. For instance, each autoequivalence
$\sigma$ of $\Der(\coh\XX)$, when applied to the canonical tilting
bundle $T_{\can}=\bigoplus_{0\leq\vx\leq\vc}\Oo(\vx)$, yields tubular
distance $$\DLF{\sigma\Oo}{\sigma\Oo(\vc)}=\ovp,$$ while $|\mu
(\sigma\Oo(\vc))-\mu (\sigma\Oo)|>0$ can get arbitrarily small, see  Theorem~\ref{thm:min_width}, (ii).

Assume $\XX$ is tubular, and $T$ is a multiplicity-free tilting sheaf whose
indecomposable summands $T_1,\dots,T_n$ have monotonically increasing
slope (with equality allowed). The question arises whether
$\DLF{T_1}{T_n}=\ovp$ characterizes $T_{\can}$ up to autoequivalence
of $\Der(\coh\XX)$. We note (without proof) that this is indeed the
case for tubular type $(2,2,2,2)$, but fails for the tubular weight triples
$(3,3,3)$, $(2,4,4)$ and $(2,3,6)$, as the following examples of
Coxeter-Dynkin algebras of canonical type show.

\def\tcddd{\xymatrix@R10pt@C20pt{ 
    &&\frac{0}{1}\\ 
    &\frac{0}{2}\ar[rd]\ar[ru]& \frac{0}{1}\\ 
    \frac{-1}{2}\ar[r]\ar[ru]\ar[rd]\ar@{..}@/_1.1pc/[rr]|{\spitz{2}}&
    \frac{0}{2}\ar[r] \ar[ru]&\frac{1}{1}\\ 
    &\frac{0}{2}\ar[ru]\ar[r]&\frac{0}{1} 
  }}

\def\tczds{\xymatrix@R10pt@C20pt{ 
    &\frac{0}{3}\ar[dr] & \frac{0}{2} & & &\\ 
    \frac{-1}{5}\ar[r]\ar[ru]\ar[rd]\ar@{..}@/_1.1pc/[rr]|{\spitz{2}}& \frac{0}{4} \ar[ur]\ar[r] &
    \frac{1}{1} & & &\\ 
    & \frac{0}{5}\ar[r]\ar[ur] & \frac{0}{4}\ar[r] & \frac{0}{3}\ar[r] &
    \frac{0}{2}\ar[r] & \frac{0}{1} 
  }}

\def\tczdsa{\xymatrix@R10pt@C20pt{ 
    &\frac{0}{3}\ar[dr] & \frac{0}{2} \\ 
    \frac{-1}{5}\ar[r]\ar[ru]\ar[rd]\ar@{..}@/_1.1pc/[rr]|{\spitz{2}}& \frac{0}{4} \ar[ur]\ar[r] &
    \frac{1}{1} \\ 
    & \frac{0}{5}\ar[d]\ar[ur]\\
    &\frac{0}{4}\ar[d]\\
    &\frac{0}{3}\ar[r] &\frac{0}{2}\ar[r] & \frac{0}{1} 
  }}

\def\tczvv{\xymatrix@R10pt@C20pt{ 
    &\frac{0}{2}\ar[dr] & \frac{0}{2}\ar[r]  & \frac{0}{1} \\ 
    \frac{-1}{3}\ar[r]\ar[ru]\ar[rd]\ar@{..}@/_1.1pc/[rr]|{\spitz{2}}& \frac{0}{3} \ar[ur]\ar[r] &
    \frac{1}{1} & \\ 
    & \frac{0}{3}\ar[ru]\ar[r] & \frac{0}{2}\ar[r] & \frac{0}{1} 
  }}

\begin{figure}[H]
$$
\begin{array}{ccc}
\tczdsa &\tcddd&\tczvv\\
(2,3,6)&(3,3,3)&(2,4,4)\\
 \end{array}
$$
\caption{Coxeter-Dynkin algebras with $\DLF{T_1}{T_n}=\ovp$}
\label{fig:tubular_counter}
\end{figure}
Note that these
algebras are Schurian and that the relations are given by equation
\eqref{eq:cox_dynkin}; moreover, they all have tubular width
$\DLF{T_1}{T_n}=\ovp$. Labels at vertices display the pair
(degree,rank) as `fractions'. We remark further that a Coxeter-Dynkin
algebra of type $(2,2,2,2)$ is isomorphic to the canonical algebra of
the same type, so it does not qualify as a (counter)-example in the
present context.
\section{Two instructive examples} \label{sect:instructive}
First we present two concealed-canonical algebras $A$ and $B$, one tubular, the other one wild, with interesting properties. We note that the quivers of $A$ and $B$ have a unique sink and a unique
source.

\begin{exa} \label{ex:instructiveI}
This example is the endomorphism ring of a tilting bundle $T$ on a
weighted projective line $\XX$ of tubular type $(3,3,3)$. Figure \ref{fig:ex0} shows a branch enlargement $A$ of
a canonical algebra of type $(2,3,3)$. The pair $(\deg E,\rk E)$ for each indecomposable
summand $E$ of $T$ is displayed in the figure below, and also later, as the (unreduced) fraction degree/rank.
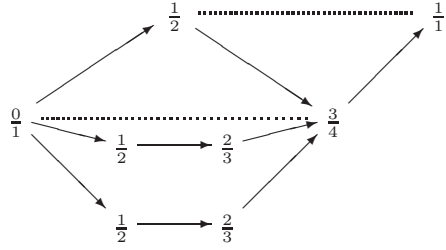
\begin{figure}[H]
\centering
\begin{picture}(175,90)
\put(5,45){
	\put(0,0){\HVCenter{\small $\frac{0}{1}$}}
	\put(40,-10){\HVCenter{\small $\frac{1}{2}$}}
	\put(40,-40){\HVCenter{\small $\frac{1}{2}$}}
	\put(80,-10){\HVCenter{\small $\frac{2}{3}$}}
	\put(80,-40){\HVCenter{\small $\frac{2}{3}$}}
	\put(60,40){\HVCenter{\small $\frac{1}{2}$}}
	\put(120,0){\HVCenter{\small $\frac{3}{4}$}}
	\put(160,40){\HVCenter{\small $\frac{1}{1}$}}
	\put(8,6){\vector(3,2){44}}
	\put(6,-1.5){\vector(4,-1){28}}
	\put(6,-6){\vector(1,-1){28}}
	\put(46,-10){\vector(1,0){28}}
	\put(46,-40){\vector(1,0){28}}
	\put(68,34){\vector(3,-2){44}}
	\put(86,-8.5){\vector(4,1){28}}
	\put(86,-34){\vector(1,1){28}}
	\put(126,6){\vector(1,1){28}}
	{\thicklines%
\qbezier[41](69.5,40)(110,40)(149.5,40)%
\qbezier[41](10,0)(50,0)(110,0)%
}%
}
\end{picture}
\caption{The algebra $A$ of tubular weight type $(3,3,3)$.}
\label{fig:ex0}
\end{figure}
We note that for each indecomposable summand $E$ of $T$ the degree-rank pair
$(\deg\,E, \rk E)$ is coprime. By~\cite{Lenzing:Meltzer:1993} this implies
that $E$ is quasi-simple in its tube which has (the maximal possible) $\tau$-period $3$.
This in turn implies that for any two indecomposable summands $T'$ and $T''$ of $T$
there exists a self-equivalence $u$ of the triangulated category $\Der(\coh\XX)$ sending
$T'$ to $T''$. To phrase it differently, the tilting bundle $T$ is homogeneous in
$\Der(\coh\XX)$. But $\End(T)$ is not a canonical algebra, implying by Theorem~\ref{thm:main0}
that there is no self-equivalence $v$ of $\Der(\coh\XX)$ such that $v(T)$ is a direct sum of
line bundles.
\end{exa}

\begin{exa}
  For weight type $(2,3,7)$ where $\chi_\XX<0$ there exists a tilting
  bundle $T=\bigoplus_{i=1}^{11} T_i$ whose indecomposable summands
  $T_i$ have rank and degree as shown in the next figure. Vertices are
  numbered [1] to [11], (unreduced) fractions $\frac{d}{r}$ represent
  the pair (degree,rank).  The quiver $Q$ and the (minimal) numbers of
  relations for the endomorphism algebra $B=\End(T)$ are displayed by
  Figure \ref{fig:ex2}.
\begin{figure}[H]
\centering
\unitlength 1.6pt 
\newcommand{\rel}[2]{{\color{white}\put(#1){\circle*{9}}}\put(#1){\HVCenter{\scriptsize$\langle #2\rangle$}}}
\newcommand{\labl}[1]{\tiny{[#1]}}
\begin{picture}(210,110)
\put(5,45){
	\put(0,0){\HVCenter{\small $\frac{0}{4}$\labl{1}}}
	\put(40,20){\HVCenter{\small $\frac{5}{26}$\labl{2}}}
	\put(60,-40){\HVCenter{\small $\frac{4}{31}$\labl{3}}}
	\put(80,0){\HVCenter{\small $\frac{18}{84}$\labl{4}}}
	\put(120,0){\HVCenter{\small $\frac{21}{85}$\labl{6}}}
	\put(100,35){\HVCenter{\small $\frac{2}{20}$\labl{5}}}
	\put(140,-40){\HVCenter{\small $\frac{6}{24}$\labl{7}}}
	\put(160,20){\HVCenter{\small $\frac{5}{14}$\labl{10}}}
	\put(150,-10){\HVCenter{\small $\frac{3}{19}$\labl{8}}}
	\put(140,60){\HVCenter{\small $\frac{13}{49}$\labl{9}}}
	\put(200,0){\HVCenter{\small $\frac{1}{1}$\labl{11}}}
	\put(6,3){\vector(2,1){28}}
	\put(6,-4){\vector(3,-2){48}}
	\multiput(46.6,18.2)(-1.2,-2.4){2}{\vector(2,-1){28}}
	\put(42,14){\vector(1,-3){16}}
	\put(63,-34){\vector(1,2){14}}
	\put(86,0){\vector(1,0){28}}
	\put(46,21.5){\vector(4,1){48}}
	\put(106,33.5){\vector(4,-1){48}}
	\put(126,3){\vector(2,1){28}}
	\put(152,-4){\vector(1,3){6}}
	\put(142,-34){\vector(1,3){6}}
	\put(123,-6){\vector(1,-2){14}}
	\put(122,6){\vector(1,3){16}}
	\multiput(141,53)(2,1){3}{\vector(1,-2){14}}
	\put(147,54){\vector(1,-1){48}}
	\put(146,-34){\vector(3,2){48}}
	\put(166,16){\vector(2,-1){28}}
	{\thicklines%
	\qbezier[31](70,-40)(100,-40)(130,-40)
	\qbezier[31](130,0)(160,0)(190,0)
	\qbezier[51](50,20)(100,20)(150,20)
	 \qbezier[46](48.57492926,8.85504245)(80,-25)(131.4250707,-36.85504245)
	 \qbezier[46](151.4250707,8.85504245)(120,-25)(69.07106781,-36.85504245)
	\qbezier[40](47,27)(80,55)(130,58)
	\qbezier[60](7,7)(50,62)(130,62)
	\rel{100,21}{7}
	\rel{115,-19}{2}
	}
}
\end{picture}
\caption{Algebra $B$ of weight type $(2,3,7)$.}
\label{fig:ex2}
\end{figure}
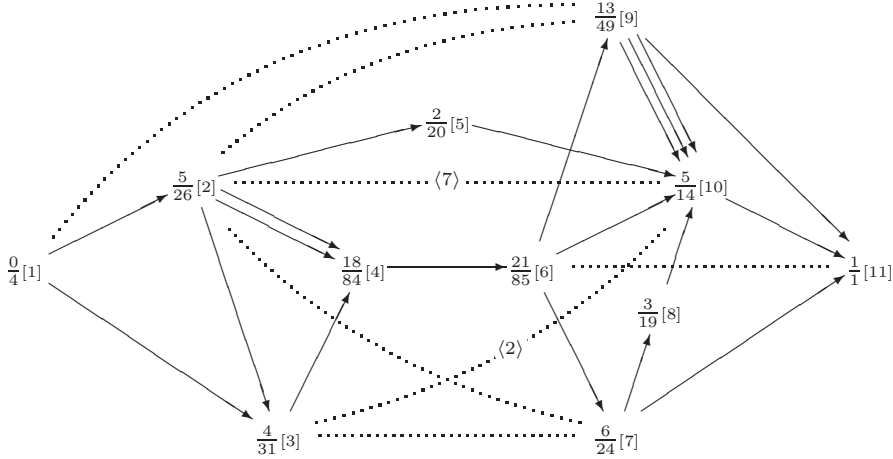

\noindent The most efficient way to construct tilting sheaves $T$ as
above is to apply H\"{u}bner reflections to the canonical configuration
$T_{\can}$, see Section~\ref{ssect:Huebner}. Here, one gets back from $T$  to $T_{\can}$, up to line bundle twist, by successive mutations in the following
vertices $6,4,9,8,7,8,3,5,1,2,10,5,9,10,7$, $3$, $1$, $10$, $8$, $2$, $9,4,9,7,6,4$.
Because we are dealing with $\leq 3$ weights, by \cite{Lenzing:Meltzer:2002} there exists, up to
isomorphism, a unique endomorphism algebra $B$ of a tilting bundle $T$
with the given quiver and number of relations.
\end{exa}

C.M.~Ringel has collected in \cite{Ringel:2009} an impressive list of properties distinguishing
canonical algebras within the class of tame concealed algebras, that is, the endomorphism rings of tilting bundles for a weighted projective line of Euler characteristic $\chi_\XX>0$.
A number of these properties relies on an inspection of the
Happel-Vossieck list classifying the tame concealed
algebras~\cite{Happel:Vossieck:1983}.

In addition to the characterizing properties from Theorem~\ref{thm:linebdle_max}, Theorem~\ref{thm:bij_max} and Theorem~\ref{thm:simples_max}, Ringel states in \cite{Ringel:2009} that for a tame concealed algebra $A$
(usually assumed to be not of type $(p,q)$) each condition
of the following list implies that $A$ is canonical:
\begin{enumerate}
\item
$A$ has only one source and one sink.
\item
$A$ is not Schurian.
\item
There exists a $2$-Kronecker pair $(X,Y)$ with $X$ simple in $\mod{A}$.
\item
There exists a $2$-Kronecker pair $(X,Y)$ with $Y$ simple in $\mod{A}$.
\item
There exists a one-parameter family of local modules.
\item
There are local modules with self extensions.
\item
There exists a one-parameter family of colocal modules.
\item
There are colocal modules with self extensions.
\item
There exists a projective indecomposable which is not thin.
\item
There exists a injective indecomposable which is not thin.
\end{enumerate}
Here, a pair $(X,Y)$ is called a
\emph{$2$-Kronecker pair}
if $X$, $Y$ are
exceptional, Hom-orthogonal, and
with an extension space $\Ext^1(Y,X)$ of dimension two. An $A$-module $X$ is called
\emph{local}, respectively \emph{colocal}, if it has a unique maximal
submodule (resp.\ a unique simple submodule). An $A$-module $X\neq0$
is called \emph{thin} if for each indecomposable projective $P$ the
space $\Hom_A(P,X)$ has dimension at most one.

As shown by our next result, characterizations of canonical algebras within the class of tame concealed algebras  have a tendency not to extend to the case of concealed-canonical algebra in general, the major exceptions to this rule being those characterizations treated in Section~\ref{sect:max_conditions}.

\begin{prop}
None of the conditions (1)--(10) yields a characterization for canonical
algebras in general.
\end{prop}

\begin{proof}
  Both algebras, $A$ and $B$, have only one source and only one sink
  and they are not Schurian, and they satisfy the conditions (1), (2),
  (9) and (10).  We now show that $B$ satisfies condition (3): Let
  $X=S_3$ be the simple associated to vertex $3$ and $Y$ be the
  $2$-dimensional indecomposable with top $S_1$ and socle $S_2$. Then
  $X$ and $Y$ are exceptional objects which are Hom-orthogonal with
  $\dim_k\Ext^1(X,Y)=2$.  For (4) repeat dualizing (3). For (5), (6),
  (7) and (8) we look at the $A$-modules given as representations in
  Figure~\ref{fig:local}. Note that the family is given by pairwise
  non-isomorphic indecomposables which are local and colocal and have
  self-extensions.\end{proof}

\begin{figure}[H]
\centering
\begin{picture}(175,95)
\put(5,50){
\put(0,0){\HVCenter{\small $k$}}
\put(40,-10){\HVCenter{\small $k$}}
\put(40,-40){\HVCenter{\small $k$}}
\put(80,-10){\HVCenter{\small $k$}}
\put(80,-40){\HVCenter{\small $k$}}
\put(60,40){\HVCenter{\small $k$}}
\put(120,0){\HVCenter{\small $k$}}
\put(160,40){\HVCenter{\small $0$}}
\multiput(7.072427502,3.813563849)(-0.832050294,1.248075442){2}{\line(3,2){46.68719529}}
\multiput(5.126524164,-6.187184335)(1.060660172,1.060660172){2}{\line(1,-1){28.6862915}}
\multiput(7.579238282,-2.667891875)(0.363803438,1.45521375){2}{\line(4,-1){24.47772}}
\multiput(113.7596228,5.061639291)(-0.832050294,-1.248075442){2}{\line(-3,2){46.68719529}}
\multiput(113.8128157,-5.126524164)(1.060660172,-1.060660172){2}{\line(-1,-1){28.6862915}}
\multiput(112.0569583,-1.212678125)(0.363803438,-1.45521375){2}{\line(-4,-1){24.47772}}
\put(46,-10){\vector(1,0){28}}
\put(46,-40){\vector(1,0){28}}
\put(126,6){\vector(1,1){28}}
{\thicklines%
\qbezier[41](69.5,40)(110,40)(149.5,40)%
\qbezier[41](10,0)(50,0)(110,0)%
}%
\put(60,-16){\HVCenter{\small $\lambda-1$}}
\put(60,-46){\HVCenter{\small $\lambda$}}
}%
\end{picture}
\caption{Distinct local, colocal and selfextending $A$-modules.}
\label{fig:local}
\end{figure}
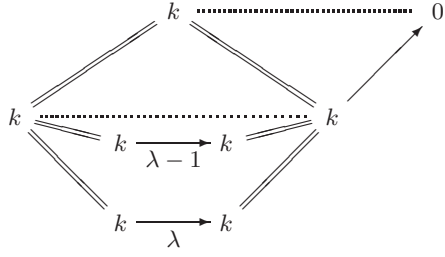

\section{Algebras antipodal to canonical} \label{sect:min_condition}
Instead of maximality properties, as studied in Section~\ref{sect:max_conditions}, we now investigate the corresponding minimality properties. We start with a couple of properties of general interest.
\subsection*{Useful generalities}
\begin{prop} \label{prop:vect+line}
Let $\XX$ be a weighted projective line,  $T$ be a tilting bundle and $L$ be a line bundle on $\XX$. Then either $\Hom(T,L)=0$ or $\Ext^1(T,L)=0$.
\end{prop}
\begin{proof}
Assume that $\Hom(T,L)\neq0$ and $\Ext^1(T,L)\neq0$. Invoking Serre duality, we obtain non-zero morphisms $u:T\ra L$ and $v:L\ra T(\vom)$; moreover $v$ is a monomorphism since $T$ is a vector bundle. Thus $vu$ is non-zero in $\Hom(T,T(\vom))=\dual\Ext^1(T,T)=0$, contradicting that $T$ is tilting.
\end{proof}

The next result is due to T.~H\"{u}bner \cite{Huebner:Diplom}, see also \cite[Proposition~6.5]{Lenzing:Reiten:2006}. It will play a central role when investigating minimality properties for positive Euler characteristic.
\begin{prop}[H\"{u}bner] \label{prop:slope+slice} Let $\XX$ be a
  weighted projective line with $\chi_\XX>0$. Then the direct sum of
  (a representative system of) the indecomposable vector bundles $E$
  with slope in the range $0\leq \mu E <|\delta(\vom)|$ is a tilting
  bundle $T_\her$ whose endomorphism ring $A$ is hereditary. Moreover,
  the following holds:
\begin{enumerate}[\upshape (i)]
  \item If $t(\XX)=3$, then each indecomposable summand $E$ of $T_\her$
  has slope $0$ or $|\delta(\vom)|/2$. Correspondingly, each vertex in
  the quiver of $A$ is a sink or a source.

  \item If $\XX$ has weight type $(p_1,p_2)$, $1\leq p_1\leq p_2$ then
  $T_\her$ is the direct sum of all line bundles $\Oo(\vx)$ with
  degree in the range $0\leq \delta(\vx)\leq
  |\delta(\vom)|-1=\delta(\vx_1)+\delta(\vx_2)-1$. The quiver of $A$
  has bipartite orientation if and only if $p_1=p_2$.~\hfill$\square$
\end{enumerate}
\end{prop}
The next result is a reformulation of a result by Kerner and
Skowro{\'n}ski \cite[Theorem~3]{Kerner:Skowronski:2001}.
\begin{thm}[Kerner-Skowro{\'n}ski]\label{thm:Kerner:Skow}
  Let $\XX$ be a weighted projective line of negative Euler
  characteristic. Further let $m$ be a positive integer. Then there
  exists a tilting bundle $T$ on $\XX$ such that for each
  indecomposable summand $T_i$ of $T$ and each simple sheaf $S$ on
  $\XX$ the space $\Hom(T_i,S)$ has dimension $\geq m$. In particular,
  each $T_i$ has rank $\geq m$.~\hfill\qed
\end{thm}
For a related but different result we refer to Section~\ref{ssect:min_width}.
\subsection*{Minimal number of line bundle summands}
In this section we investigate the number of non-isomorphic line bundle summands of a tilting bundle $T$ on $\XX$. Note that the index $[\LL:\ZZ\vom]$ equals the number of Auslander-Reiten orbits of line bundles. It is not difficult to see that $\pm[\LL:\ZZ\vom]=p_1\ldots p_t\chi_\XX$, where the number on the right hand side is known as the \emph{Gorenstein invariant} or \emph{Gorenstein parameter} of the $\LL$-graded coordinate algebra $S=S(\pp,\lala)$ of $\XX$. For positive Euler characteristic, we obtain the following values for $[\LL:\ZZ\vom]$:
\begin{center}
\begin{tabular}{|c|c|c|c|c|c|}\hline
weight type & $(p_1,p_2)$ & $(2,2,n)$ & $(2,3,3)$ & $(2,3,4)$ & $(2,3,5)$\\\hline
$[\LL:\ZZ\vom]$& $p_1+p_2$& $4$& $3$& $2$& $1$\\ \hline
\end{tabular}
\end{center}

\begin{prop}\label{prop:min_number_lbundle}
  $\mathrm{(i)}$ Assume $\chi_\XX>0$. Then each tilting bundle $T$ on $\XX$
  contains at least one member from each Auslander-Reiten orbit of
  line bundles.  In particular, $T$ contains at least $[\LL:\ZZ\vom]$
  non-isomorphic line bundles. This minimal value is attained if
  $\End(T)$ is hereditary.

$\mathrm{(ii)}$ Assume $\chi_\XX\leq0$. Then there exist a tilting bundle on $\XX$ without a line bundle summand.
\end{prop}
Note that the converse of the last statement of assertion (i) is not true. For weight type $(2,3,5)$ there exists a tilting bundle $T$ with endomorphism ring and rank distribution as follows:
$$
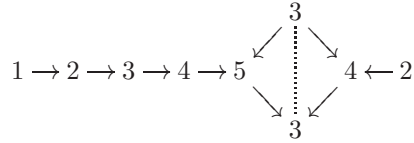
\xymatrix@C10pt@R10pt{
&&&&&3\ar[dl]\ar[dr]\ar@{..}[dd]&&\\
1\ar[r]&2\ar[r]&3\ar[r]&4\ar[r]&5\ar[dr]&&4\ar[dl]&2\ar[l]\\
&&&&&3&&\\
}
$$

\begin{proof}[Proof of the proposition]
We first assume $\chi_\XX>0$. Given a line bundle $L_0$, we choose a line bundle $L=L_0(n\vom)$, $n\in\ZZ$, such that
$$
\text{(a) } \Hom(T,L)\neq0,\text{ and}\quad \text{(b)} \Hom(T,L(\vom))=0.
$$
This choice is possible since $\delta(\vom)<0$. Now, (b) expresses
that $\Ext^1(L,T)=0$, while (a) implies in view of
Proposition~\ref{prop:vect+line} that $\Ext^1(T,L)=0$. Altogether,
$T\oplus L$ has no self-extensions, implying that $L$ is a direct
summand of $T$, since $T$ is tilting. This shows the first claim of
assertion (i). Further the tilting bundle $T_\her$ of
Proposition~\ref{prop:slope+slice} contains exactly one member from
each Auslander-Reiten orbit of an indecomposable vector bundle, hence
in particular, $T_\her$ contains exactly $[\LL:\ZZ\vom]$
non-isomorphic line bundle summands. Since tilting bundles $T$ with
hereditary endomorphism ring form a slice in the Auslander-Reiten
quiver of $\vect\XX$, the same argument applies in this case.

We now assume $\chi_\XX=0$.  By ~\cite{Lenzing:Meltzer:2000}
there exists an autoequivalence $\rho$ of $\Der (\coh\XX)$ such that
the induced map on pairs $(\deg X,\rk X)^t$ is given by left multiplication with the matrix
$$
\left(
\begin{array}{cc}
1& 0\\
1&1
\end{array}
\right).
$$ Let $T=\rho (T_{\can}(\vu))$ and
$A=\End(T)$ where $\vu$ has degree one. Note that $A$ is a canonical algebra; moreover the degree/rank distribution for the indecomposable summands of $T_\can(\vu)$ along the $i$th arm of the quiver of the canonical algebra $A$ is given by
$$
\frac{1}{1}\ra \frac{1+\ovp/p_i}{1}\ra \frac{1+2\ovp/p_i}{1}\ra \cdots \ra \frac{1+\ovp}{1}.
$$
Applying $\rho$ we obtain the corresponding degree/rank distribution for the indecomposables of the $i$th arm of $T$ as
$$
 \frac{1}{2}\ra \frac{1+\ovp/p_i}{2+\ovp/p_i}\ra \frac{1+2\ovp/p_i}{2+2\ovp/p_i}\ra \cdots \ra \frac{1+\ovp}{2+\ovp}.
$$
It follows that all ranks for the indecomposables in the $i$th arm of $T$ have rank $\geq2$, so that the claim follows.

Finally assume that $\chi_\XX<0$. Then the claim follows from Theorem~\ref{thm:Kerner:Skow} or Theorem~\ref{thm:min_width}.
\end{proof}

\subsection*{Minimal number of bijections}\label{ssect:min_bij}

\subsubsection*{Positive Euler characteristic}
Here the following cases arise:

(a) Assume weight type $(2,3,p)$ with $p=3,4,5$ and consider the
tilting bundle $T_\her$ from Proposition~\ref{prop:slope+slice}. For
each arrow from the quiver $Q$ of $\End(T)$ the source and sink have
different rank. Accordingly there are no arrows $u\ra v$ inducing a
bijection $G_v\ra G_u$ for the generic $\End(T)$-module $G$.

(b) Assume weight type $(2,2,p)$ with
$p\geq2$. Invoking~\cite{Happel:Vossieck:1983} it follows that $p-2$
is the minimal number of arrows inducing a bijection. This number is
attained for the tilting bundle $T_\her$ from
Proposition~\ref{prop:slope+slice}.

(c) Assume weight type $(p_1,p_2)$ with $1\leq p_1\leq p_2$, and let
$T$ be any tilting bundle. Then the quiver $Q$ of $\End(T)$ has
$n=p_1+p_2$ vertices and also $n$ arrows. Since all indecomposable
summands of $T$ have rank one, each arrow $u\ra v$ of $Q$ induces a
bijection $G_v\ra G_u$.
\subsubsection*{Euler characteristic zero}
We have shown in the proof of Proposition~\ref{prop:min_number_lbundle} that there exists a tilting bundle $T$ whose endomorphism ring is the canonical algebra and such that the degree/rank distribution in the $i$th arm is given by
$$
 \frac{1}{2}\ra \frac{1+\ovp/p_i}{2+\ovp/p_i}\ra \frac{1+2\ovp/p_i}{2+2\ovp/p_i}\ra \cdots \ra \frac{1+\ovp}{2+\ovp}.
$$
It follows that all ranks for the indecomposables in the $i$th arm are pairwise distinct. Hence no arrow $u\ra v$ induces a bijection $G_v\ra G_u$ for the $T$-distinguished generic $\End(T)$-module $G$.
\subsubsection*{Negative Euler characteristic}
For the minimal wild types $(3,3,4)$, $(2,4,5)$ and $(2,2,2,2,2)$, the degree/rank data for the tilting bundles $T$ of Figure~\ref{fig:min_wild} show that no arrow $u\ra V$ of $\End(T)$ induces a bijection $G_v\ra G_u$. For weight type $(2,3,7)$ the same conclusion follows by inspection of Figure~\ref{fig:ex2}. Finally, for weight type $(2,2,2,3)$ we modify the example from Figure~\ref{fig:min_wild} by H\"{u}bner reflection in the sink $[7]$ yielding an example with the wanted properties.

\subsection*{Minimal number of central simple modules}
Let $T$ be a tilting bundle on $\XX$ with endomorphism ring $A$. Recall that we identify $\mod{A}$ with a full subcategory of $\Der(\coh\XX)$ and call a simple $A$-module $S$ \emph{central simple} if $S$ has rank zero, that is, belongs to $\cohnull\XX$.
\begin{prop} Depending on the Euler characteristic, the following properties hold.
\begin{enumerate}[\upshape (i)]
\item Assume $\chi_\XX >0$. Then the following assertions hold.
\begin{itemize}
\item[(a)] If $t(\XX)=3$ there exists a tilting bundle $T_{\her}$ with a
  hereditary endomorphism ring $A$ and without central simple
  $A$-modules.
\item[(b)] Assume weight type $(p_1,p_2)$ with $1\leq p_1 \leq p_2$. Then for
  each tilting bundle $T$ the endomorphism ring $A$ has at least $p_2 -p_1$
  central simple $A$-modules, and this bound is attained.
\end{itemize}
\item Assume $\chi_\XX=0$. Then there is a tilting bundle $T$ such that its endomorphism ring
  $A$ is canonical without central simple modules.

\item Assume $\chi_\XX<0$. Then there exists a tilting bundle $T$ such that its endomorphism ring $A$ has no central simple modules.
\end{enumerate}
\end{prop}
\begin{proof}
  {Case (i)(a)}: The tilting bundle $T_\her$ from
  Proposition~\ref{prop:slope+slice} has an endomorphism ring $A$
  whose quiver has bipartite orientation. Let $T_1,\ldots,T_n$ denote the (pairwise non-isomorphic) indecomposable summands of $T$. Thus the simple $A$-module
  $S_i$ attached to $T_i$ equals $T_i$ (resp.\ $\tau T_i[1]$) if $i$ is a
  source (resp.~a sink)
  of the quiver of $A$. In particular, each $S_i$ has non-zero rank,
  and $A$ has no central simple modules.

{Case (i)(b)}: We refer to Lemma~\ref{lem:two_weights}, proved below.

{Case (ii)}: We thus consider the case where $\XX$ is tubular. By ~\cite{Lenzing:Meltzer:2000}
there exists an autoequivalence $\rho$ of $\Der (\coh\XX)$ such that
the induced map on slopes is $q\mapsto\frac{q}{1+q}$. Let $T=\rho T_{\can}$ and
$A=\End(T)$. Then the simple $A$-modules all have slope $0$, $1$, or
$\frac{\ovp}{1+\ovp}$. Hence none of these has rank zero.

{Case (iii)}: Assume finally that $\chi_\XX <0$. By Theorem~\ref{thm:Kerner:Skow} there exist infinitely many tilting bundles $T$ such that $\End(T)$ has no central simples.
\end{proof}
\def\zvf{\xymatrix@C7pt@R7pt{ 
&&&\frac{1}{1}[2]&&\frac{9}{2}[4]\\
&&\frac{3}{2}[1]\ar[ru]&&\frac{0}{1}[7]\ar[ru]\ar[rd]\\
&\frac{6}{3}[10]\ar@{..}[rrrruu]\ar[ru]\ar[rd]&&\frac{19}{7}[9]\ar@{..}[rr]\ar[ru]\ar[rd]&&\frac{11}{3}[3]\ar[uu]\\
\frac{4}{2}[5]\ar@{..}[rrrrruuu]\ar[rr]&&\frac{4}{1}[8]\ar[ru]&&\frac{14}{5}[6]\ar[ru]\\
}}

\def\zds{\xymatrix@C10pt@R10pt{ 
&\frac{16}{1}[2]&&&&\frac{6}{1}[6]\\
\frac{13}{2}[8]\ar@{..}[ru]\ar[rd]&&\frac{46}{5}[7]\ar[lu]&&\frac{30}{4}[1]\ar[ru]&&\\
 &\frac{40}{5}[10]\ar[ru]\ar[ld]&&\frac{40}{5}[11]\ar[lu]\ar[ru]&&\frac{17}{2}[4]\ar@{..}[uu]\ar[lu]\ar[dr]\\
\frac{19}{2}[9]&&\frac{0}{1}[5]\ar@{..}[luuu]\ar@{..}[uu]\ar@{..}[ll]\ar[lu]\ar[ru]&&&&\frac{8}{1}[3]\\
}}

\def\ddv{\xymatrix@R20pt@C20pt{   
\frac{1}{2}[1]&\frac{4}{4}[9]\ar[l]\ar[d]&\frac{1}{3}[3]\ar@{..}[ld]\ar@{..}[rd]|{\langle2\rangle}\ar[l]\ar[r]\ar[d]&\frac{2}{2}[4]\ar[d]\\
&\frac{5}{3}[2]\ar@/^1.3pc/[rr]&\frac{6}{5}[6]\ar[l]\ar[r]&\frac{4}{1}[8]\\
&\frac{1}{1}[7]\ar[ru]\ar@{..}[u]&&\frac{2}{2}[5]\ar[lu]\ar@{..}[u]\\
}}

\def\zzzd{\xymatrix@R8pt@C20pt{ 
&\frac{4}{2}[2]\ar[rdd]^{b_5}&\\
&\frac{4}{2}[3]\ar[rd]_{b_4}&\\
\frac{0}{1}[1]\ar[ruu]^{a_5}\ar[ru]^{a_4}\ar[rd]^{a_3}\ar@<.5ex>[rdd]_{a_2}\ar@<-.5ex>[rdd]^{a_1}&&\frac{7}{3}[7]\\
&\frac{4}{2}[4]^{b_3}\ar[ru]&\\
&\frac{5}{3}[6]\ar[ruu]_{b_3}\ar[r]^{c}&\frac{2}{1}[5]\\
}}
\def\relzzzd{$b_3a_3=ba_1$,  $b_4a_4=ba_2$, $b_5a_5=b(a_2-a_1)$, $c(a_2-\la a_1)=0$}
\def\relzzzzz{$b_1a_i=\la_ib_1a_1$ for $i=3,4,5$; $b_2a_i=b_2a_j$ for $i,j=3,4,5$; $b_ja_i=0$ for $j\neq1,2,i$}

\def\zzzzz{\xymatrix@R10pt@C30pt{ 
&&\frac{1}{1}[2]\\
&&\frac{1}{1}[3]\\
\frac{0}{1}[1]\ar[r]^{a_2}\ar@<2ex>[r]^{a_1}
\ar@<-2ex>[r]^{a_3}&\frac{3}{4}[7]\ar[ruu]^{b_1}\ar[ru]^{b_2}\ar[r]^{b_3}\ar[rd]_{b_4}\ar[rdd]_{b_5}&\frac{1}{1}[4]\\
&&\frac{1}{1}[5]\\
&&\frac{1}{1}[6]\\
}}

For illustration, we present explicit examples for the minimal wild
weight types in Figure~\ref{fig:min_wild}. For the three algebras of
triple weight type the graphical information determines the algebras
up to isomorphism, see~\cite{Lenzing:Meltzer:2002}. For the two remaining weight types, the explicit relations are given afterwards.

\begin{figure}[pb]
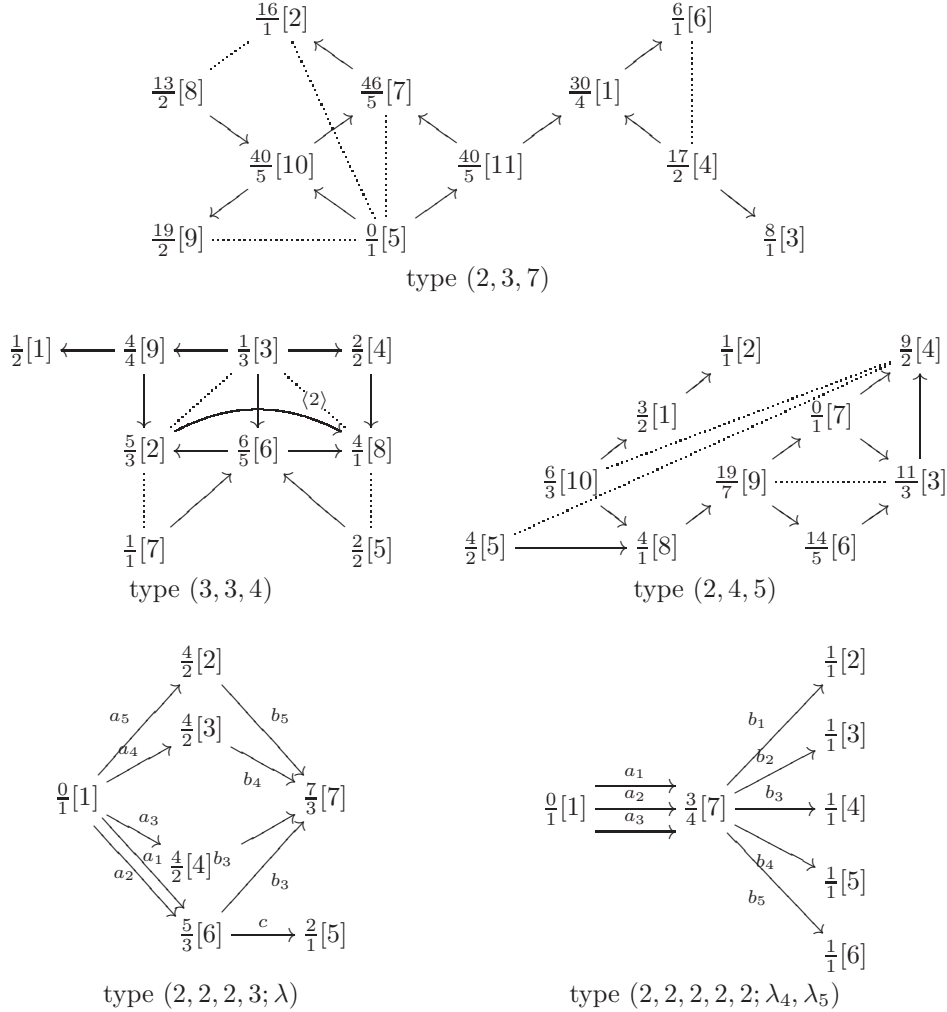

\begin{tabular}{ccc}
\multicolumn{3}{c}{$\zds$}\\
\multicolumn{3}{c}{type $(2,3,7)$}\\
&\\
$\ddv$&\quad&$\zvf$\\
type $(3,3,4)$ && type $(2,4,5)$\\
&\\
$\zzzd$ &\quad& $\zzzzz$\\
type $(2,2,2,3;\la)$ && type $(2,2,2,2,2;\la_4,\la_5)$\\
\end{tabular}
\caption{Minimal wild canonical type without central simple modules}
\label{fig:min_wild}
\end{figure}
\sloppy
The following sequences of H\"{u}bner reflections, see Section~\ref{ssect:Huebner}, transform the tilting bundles, depicted above, into $T_\can$, up to line bundle twist: $(2,3,7)$: (10, 7, 11, 1, 2, 4, 7, 8, 9, 10, 11);  $(2,4,5)$: $(9,6,4,5,3,2,10,1)$;  $(3,3,4)$: $(6,5,4,3,2,9,1)$; $(2,2,2,3)$: $(6,4,3,2,7)$; $(2,2,2,2,2)$: (7).
\fussy
Concerning $(2,2,2,3)$, we impose the relations \relzzzd\ where $\la$ is supposed to be different from $0,1$. Concerning $(2,2,2,2,2)$, we impose the relations \relzzzzz. We assume $\la_3=1$ and $\la_4\neq\la_5$ to be different from $0,1$.

\begin{lem} \label{lem:two_weights}
Assume $\XX$ of weight type $(p_1,p_2)$. Let $T$ be a tilting bundle and $Q$ the quiver of $\End(T)$. Then the number $\nu(T)$ of central simple $A$-modules equals the number of vertices of $Q$ which are neither a sink nor a source. Always we have $\nu(T)\geq|p_1-p_2|$ with equality attained for the tilting object $T$ given by the scheme
$$
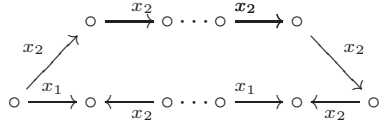
\xymatrix@C18pt@R20pt{
                 &\circ\ar[r]^{x_2}&\circ\ar[r]^{x_2}\cdots\circ\ar[r]^{x_2}&\circ\ar[rd]^{x_2}\\
\circ\ar[r]^{x_1}\ar[ru]^{x_2}&\circ&\circ\ar[l]^{x_2}\cdots\circ\ar[r]^{x_1}&\circ&\circ\ar[l]^{x_2}\\
}
$$
Assuming $p_1\leq p_2$, the scheme contains  $p_1$ pairs
$\xymatrix{\circ\ar[r]^{x_1}&\circ&\circ\ar[l]^{x_2}}$ of arrows, labeled $x_1$ and $x_2$, that are followed by
$p_2-p_1$ arrows, labeled $x_2$, with arrows labeled $x_2$ (resp.\ $x_1$)
having clockwise (resp.\ anticlockwise) orientation.
\end{lem}
\begin{proof}
  If $i\in[1,n]$ is a source (resp.\ sink) of $Q$, then the simple
  $A$-module $S_i$, corresponding to $T_i$, has rank $1$ (resp.\ $-1$). Assume, conversely,
  that $i$ is not a sink or a source of $Q$, hence locally we have one of the two cases
  (a)$\xymatrix{i+1\ar[r]^-{x_1}&i\ar[r]^-{x_1}&i-1}$ or (b)
  $\xymatrix{i-1\ar[r]^-{x_2}&i\ar[r]^-{x_2}&i+1}$ where we say that $i$ is an interior
   vertex. We claim that then
  $S_i$ has rank zero. Assuming case (a), let $U$ be the unique simple
  sheaf concentrated in the first exceptional point $\la_1$ having the
  additional property $\Hom(T_i,U)\neq0$ (and then
  $\Hom(T_i,U)=k$). We note that multiplication by $x_2$ (resp.\ by
  $x_1$) acts as the identity (resp.\ the zero map) on $U$. Since all
  the $x_1$-arrows of $Q$ (we have $p_1$ of them) have the same
  orientation, we conclude that $\Hom(T_j,U)=0$ for each vertex $j\neq
  i$. Under our usual identification of modules and sheaves $U$ thus
  equals the simple $A$-module $S_i$, which therefore has rank
  zero. This proves the fist claim and further shows that $\nu(T)$
  equals the number of interior vertices $i$ in the cyclic
  arrangement of labels $x_1$ and $x_2$. It follows that
  $\nu(T)\geq|p_1-p_2|$. The proof of the last claim is obvious.
\end{proof}

\subsection*{Minimal width} \label{ssect:min_width}
Tilting bundles of minimal width only exist for positive Euler characteristic as is shown
in our next result.
\begin{thm}\label{thm:min_width}
Let $\XX$ be a weighted projective line.
\begin{itemize}
\item[(i)] Assume $\chi_\XX>0$. Then the minimal width for  tilting bundles on $\XX$ equals $|\delta(\vom)|/2$ for $t(\XX)=3$ and $\delta(\vx_1)+\delta(\vx_2)$ for $t(\XX)\leq2$.
\item[(ii)] If $\chi_\XX\leq0$ then there exists a sequence $(T_n)$ of tilting bundles on $\XX$ such that the sequence
$(\width(T_n))$ converges to zero and, moreover, each indecomposable summand $E$ of $T_n$
has rank $\geq n$.
\end{itemize}
\end{thm}
\begin{proof}
Concerning (i) we use Proposition~\ref{prop:slope+slice} stating that the direct sum $T$ of (a representative system of) the indecomposable vector bundles $E$ of slope $0\leq \mu E < |\delta(\vom)|$ forms a tilting bundle. For $t(\XX)=3$, each indecomposable summand $E$ of $T$ actually has slope $0$ or $|\delta(\vom)|/2$, showing that the width of $T$ equals $|\delta(\vom)|/2$. For $t(\XX)\leq2$, each indecomposable vector bundle has rank one, such that $T$ is the direct sum of all \emph{line bundles} $\Oo(\vx)$ with $0\leq\vx\leq |\delta(\vom)|-1$. Thus in this case the width of $T$ equals $|\delta(\vom)|-1=\delta(\vx_1)+\delta(\vx_2)-1$.

In the tubular case assertion (ii) is covered by Proposition~\ref{prop:width_tubular}. For $\chi_\XX<0$ the proof of (ii) is also given afterwards.
\end{proof}

We first assume that $\YY$ is tubular, and collect some facts on the
tubular mutations $\si$ and $\rho$ from Section~\ref{ssect:tub_mut}. Let
$\bw=[S_0]$ denote the class of any ordinary simple sheaf $S_0$. Further let $\bp=\bp(\YY)$
denote the largest weight of $\YY$. We first note
that for each $y$ from $\Knull(\coh\YY)$ we have
\begin{equation}\label{eq:multiple_shift}
\si^{n\bp}(y) = y + n \rk(y)\,\bw,
\end{equation}
a formula valid for any weight type. (This follows from the formula $\bp\vx=\vc$ if $\vx$ is one of the standard generators $\vx_i$ of $\LL$ of degree one.) By means of  the conjugation formula
\eqref{eq:conj_from_braid},
$\rho^{-1}=(\rho^{-1}\si)^{-1}\si(\rho^{-1}\si)$, we obtain a corresponding formula
\begin{equation}\label{eq:mult_right_mut}
\rho^{n\bp}(y) = y + n\,\deg(y)\bz,
\end{equation}
for the action of $\rho$ on members $y$ of $\Knull(\coh\YY)$, where $\bz$ denotes the class of $Z=\si^{-1}\rho(S_0)$.

We call a bundle $E$ on a weighted projective line $\XX$ \emph{omnipresent on $\XX$}
if $\Hom(E,S)\neq0$ for each simple sheaf $S$.
\begin{lem} \label{lem:omnipresent}
The indecomposable bundle $Z=\si^{-1}\rho(S_0)$ is omnipresent on $\YY$.
\end{lem}
\begin{proof}
Let $S$ denote any simple sheaf, say of degree $d$. Then $\si S$ is again simple,
having the same degree. Since $\rho^{-1}$ acts on the degree-rank pair $(d,0)$
by right multiplication with $\left(\begin{array}{cc}1&-1\\ 0&1\end{array}\right)$,
we obtain that $(d,-d)$ is the degree-rank pair for $\rho^{-1}\si S_0$.
It follows that
$$\dim \Hom(Z,S)=\LF{Z}{S}=\LF{S_0}{\rho^{-1}\si S}=-\rk(\rho^{-1}\si S)=d>0,$$
as claimed.
\end{proof}

\begin{prop} \label{prop:width_tubular}
Assume that $\YY$ is tubular. Let $T=\bigoplus_{i=1}^mT_i$ be a tilting bundle on $\YY$
whose indecomposable summands $T_i$ all have strictly positive slope. For
each integer $n\geq0$ we put
$$
T(n)=\rho^{n\bp}T \quad \textrm{ and }\quad T_i(n)=\rho^{n\bp}T_i.
$$
Then each $T(n)$ is a tilting bundle on $\YY$ with endomorphism ring isomorphic to $\End(T)$.
Moreover, the following holds for each $i=1,\ldots,m$:

(a) We have $\rk(T_i(n)) > \bp\,n$, moreover the slope sequence $(\mu T_i(n))_n$ converges to zero. In
particular, the width sequence $(\width(T(n)))$ converges to zero.

(b) For each simple sheaf $S$ on $\YY$ we have $\dim \Hom(T_i(n),S)\geq n$.
\end{prop}
\begin{proof}
We put $d_i=\deg T_i$, $r_i=\rk T_i$ and use similarly $d_i(n)$ and $r_i(n)$ for the
degree-rank data of $T_i(n)$. Then
\begin{equation} \label{eq:mut_iterated}
(d_i(n),r_i(n))=(d_i,r_i)\left(\begin{array}{cc} 1 & n\bp\\ 0&1
\end{array}\right) = (d_i,r_i+n\bp\,d_i).
\end{equation}
By assumption we have $r_i>0$ and $d_i\geq1$, hence $\rk(T_i(n))=r_i+n\bp\,d_i>n\bp$. Moreover,
the sequence of slopes
$$
\mu(T_i(n))=\frac{d_i}{r_i+n\bp\,d_i}
$$
obviously converges to zero. This proves assertion (a).

Concerning (b), we apply formula~\eqref{eq:mult_right_mut} to the class $y=[T_i]$ and obtain
$[T_i(n)]= [T_i]+ n\,d_i[Z]$, hence
$\dim \Hom(T_i(n), S)=\LF{[T_i(n)]}{[S]}=\LF{[T_i]}{[S]}+n\,d_i\LF{[Z]}{[S]}\geq n$
where the inequality uses that $d_i\geq1$ and $Z$ is omnipresent on $\YY$ by Lemma~\ref{lem:omnipresent}.
\end{proof}

\begin{proof}[Proof of Theorem~\ref{thm:min_width}, part (ii)]
Next we assume that $\chi_\XX<0$. In several steps we are going to construct a sequence of tilting bundles
$\sT(n)$, $n\geq0$, on $\XX$ satisfying the claims of Theorem~\ref{thm:min_width}.

\emph{Step 1.} Let $\bq=(q_1,\ldots,q_s)$ be the weight type of $\XX$. After reordering we can write $\bq=\bp+\bh$ where $\bp=(p_1,\ldots,p_t,1,\ldots,1)$ is tubular and $\bh=(0,\ldots,0,h_r\ldots,h_s)$ has
entries $h_i\geq1$ for $i=r,\ldots,s$. For each of the (distinct) exceptional points $x_i$ of $\XX$ with
$i=r,\ldots,s$ we fix a linear branch $B(i)$ of length $h_i$ which is concentrated in $x_i$. Thus
$B(i)=\oplus_{j=1}^{h_i} U_j(i)$ where
\begin{equation} \label{eq:branch}
B(i):\quad U_{h_i}(i)\epi U_{h_i-1}(i)\epi \cdots \epi U_1(i)
\end{equation}
consists of a chain of finite length sheaves concentrated in $x_i$ such that each $U_j(i)$ has length $j$, and hence $U_1(i)$ is exceptional simple on $\XX$. We call $U_{h_i}(i)$ the \emph{root of $B(i)$}.
Then, putting $B=B(r)\oplus B(r+1)\oplus \cdots \oplus B(s)$, the right perpendicular category $\rperp{B}$ of $B$ in $\coh\XX$ can be identified with the category of coherent sheaves on a
weighted projective line $\YY$ having tubular type $\bp$, see~\cite{Geigle:Lenzing:1991}. Moreover, if $T$ is a tilting bundle on $\YY$, then $\bar{T}=T\oplus B$ is a tilting sheaf on $\XX$, whose bundle part `lives on' $\YY$. For further details we refer to \cite[Theorem 3.1 and Theorem 4.1]{Lenzing:Meltzer:1996}.

\emph{Step 2.} Keeping the notations of Proposition~\ref{prop:width_tubular} we extend the tilting bundles $T(n)=\rho^{n\bp} T$ on $\YY$ obtained by the preceding step to the tilting sheaf
$\bT(n)=T(n)\oplus B$ on $\XX$. Note that the embedding $\coh\YY\incl \coh\XX$ preserves the
rank but not the degree. In fact, the association $\deg_\YY Y\mapsto \deg_\XX Y$, for $Y$ in $\coh\YY$, does not extend to a mapping $\Knull(\coh\YY)\ra \Knull(\coh\XX)$, see
\cite[Section 9]{Geigle:Lenzing:1991}.  The following lemma will allow us to bypass this technical
difficulty. We note that we continue to use the notations of Proposition~\ref{prop:width_tubular}.
\begin{lem}
For each $i=1,\ldots,m$ the sequence $(\mu_\XX(T_i(n)))_n$ of slopes, formed in $\coh\XX$, converges to $\frac{1}{\bp}\deg_\XX Z$.
Here $Z=\si^{-1}\rho(S_0)$ is formed in $\coh\YY$ with $S_0$ an ordinary simple sheaf on $\YY$, and $\bp=\lcm(p_1,\ldots,p_t)$.
\end{lem}
\begin{proof}
We first note that formula~\eqref{eq:mult_right_mut} holds in $\Knull(\YY)$, hence in $\Knull(\XX)$. Thus for each $y\in \Knull(\YY)$ we have
\begin{equation}
\deg_\XX(\rho^{n\bp}(y))=\deg_\XX(y)+n\,\deg_\YY(y)\,\deg_\XX(Z).
\end{equation}
By formula~\eqref{eq:mut_iterated} we obtain $\rk(\rho^{n\bp}y)=\deg_\YY(y)+n\,\bp\,\rk(y)$. Thus the slope sequence
$$
\mu_\XX(\rho^{n\bp}y)=\frac{\deg_\XX(y)+n\,\deg_\YY(y)\,\deg_\XX(Z)}{\rk(y)+n\,\bp\,\deg_\YY(y)}
$$
converges to $\frac{1}{\bp}\deg_\XX(Z)$.
\end{proof}

\emph{Step 3.} By means of a sequence of H\"{u}bner reflections, we next transform $\bT(n)$ into a tilting bundle $\sT(n)$ on $\XX$. Let $U_h$ be the root of a branch $B=B(i)$. Then $U_h$ is a formal sink
of $\bT=\bT(n)$, and reflection at $U_h$ yields a new tilting sheaf $\bT/U_h\oplus U_h^*(n)$, where
$U_h^*(n)$ is the kernel term of the reflection sequence
\begin{equation} \label{eq:reflection_sequence}
0\lra U_h^*(n) \lra \bigoplus_{j=1}^m T_j^{\ka_j}(n)\lra U_h \lra 0,
\end{equation}
compare formula \eqref{eq:Huebner:reflection}. Since some exponent $\ka_j$ is
non-zero, we see that $U_h^*(n)$ is an exceptional vector bundle of rank $r(n)=\sum_{j=0}^m\ka_j\rk (T_j(n))>n$. We next show that the slope sequence $\mu_\XX(U_h^*(n))$ also converges to $\frac{1}{\bp}\deg_\XX(Z)$. Clearly,
$$
\mu_\XX(U_h^*(n))=\mu_\XX(\bigoplus_{j=1}^m T_j^{\ka_j}(n))-\frac{\deg_\XX(U_h)}{r(n)}.
$$
Now the first summand $\al(n)$ on the right hand side is a convex combination of the slopes
$\mu_\XX(T_j(n))$, and thus yields a sequence $(\al(n))$ converging to $\frac{1}{\bp}\deg_\XX(Z)$, while the second summand $\be(n)$ yields a sequence $(\be(n))$ converging to zero. This proves the claim for this first step. We now continue reflecting roots of branches until all branches are exhausted, the resulting sequence of tilting bundles $\sT(n)$ then satisfies all claims. This finishes the proof of Theorem~\ref{thm:min_width} (ii).
\end{proof}

To construct explicit examples, the following result is useful.

\begin{prop}\label{prop:factorization}
Assume $\YY$ is tubular, and $T=\bigoplus_{0\leq\vx\leq\vc}T_\vx$ is a tilting bundle on $\YY$
with $\End(T)$ canonical, accordingly with $\Hom(T_\vx,T_\vy)=\Hom(\Oo(\vx),\Oo(\vy))$ for all $0\leq\vx,\vy\leq\vc$. We
assume that the width $\width(T)$ of $T$ is strictly less than $\bp=\bp(\YY)$, the maximal possible one. Let further $U$
be any sheaf of finite length. Then each morphism $T_\vx\ra U$ factors through any non-zero
morphism $u:T_\vx\ra T_\vc$.
\end{prop}
We note that the assertion is wrong if $T$ attains the maximal possible width $\bp$. Indeed,
by Theorem~\ref{thm:max_width}, we then may assume that $T_\vx=\Oo(\vx)$ holds for each $\vx$.
If $S$ denotes the exceptional simple sheaf defined by exactness of $0\ra \Oo(\vc-\vx_1)\ra\Oo(\vc)\ra S\ra 0$, then $\Hom(\Oo(\vc-\vx_1),S(\vom))=k$ but we further have
$\Hom(\Oo(\vc),S(\vom))=0$.

\begin{proof}[Proof of Proposition~\ref{prop:factorization}.] Since $\End(T)$ is canonical, there
exists a self-equivalence $\phi$ of $\Der(\coh\YY)$ mapping $\Oo(\vx)$ to $T_\vx$ for each
$0\leq\vx\leq\vc$. Any triangle $\mu: T_\vx\stackrel{u}{\lra} T_\vc\ra V_\vx\ra$ is thus given
as the image under $\phi$ of a triangle represented by a short exact sequence
$\eta:0\ra \Oo(\vx)\stackrel{v}{\lra}\Oo(\vc)\ra U_\vx\ra 0$ in $\coh\XX$. Having finite length, all the $U_\vx$ have the same slope. Hence all the $V_\vx=\phi(U_\vx)$ have the same slope $q$. If $q=\infty$, then all $V_\vx$ have rank zero, implying $\rk T_\vx=\rk T_\vc$ for each $0\leq \vx\leq \vc$. Then Theorem~\ref{thm:main0} implies that $T=T_\can$ up to a line bundle twist
and hence $\width(T)=\bp$, contradicting our assumption on $T$.

Thus each $V_\vx$ is a vector bundle and thus the triangle $\mu$ yields an exact sequence
$\mu: 0\ra T_\vx\stackrel{u}{\lra}T_\vc\ra V_\vx\ra 0$ in $\coh\YY$ whose terms are vector bundles.
Since $\Ext^1(-,U)$ vanishes on $\vect\YY$ for each $U$ of finite length, the sequence
$$
0\ra \Hom(V_\vx,U)\lra \Hom(T_\vc,U)\stackrel{-\circ u}{\lra} \Hom(T_\vx,U)\ra 0
$$
is exact, thus proving the claim.
\end{proof}
We are now constructing an explicit sequence of tilting bundles $T^*(n)$ on the weighted projective line $\XX$ of weight type $(2,4,7)$ illustrating the arguments of this section.
We start with the tilting bundle $T=T_\can(\vc)$ on the tubular weighted projective line $\YY$ of
weight type $(2,4,4)$, and form the sequence of $T(n)$ of tilting bundles of Proposition~\ref{prop:width_tubular}. Fixing a branch $B: U_3\epi U_2\epi U_1$ of length $3$ concentrated in the third exceptional point of $\XX$ we identify $\coh\YY$ with the perpendicular subcategory $\rperp{B}$ in $\coh\XX$, and then enlarge $T(n)$ to the tilting sheaf $\bT(n)=T(n)\oplus B$. The endomorphism ring of $\bT(n)$ is then given by the following quiver
with relations.
$$
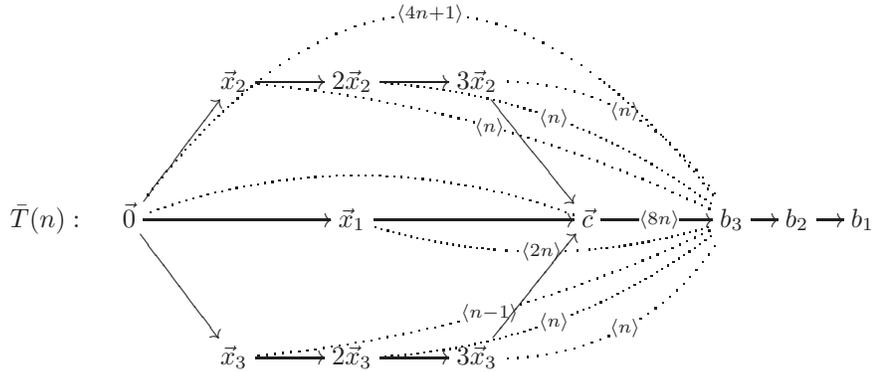
\xymatrix@C10pt@R15pt{
&&&&&&&& \\
&&&\vx_2\ar[rr]\ar@/^0.8pc/@{..}[rrrrrrrrrdd]|{\sz{n}}&&2\vx_2\ar[rr]\ar@/^1.1pc/@{..}[rrrrrrrdd]|{\sz{n}}&&3\vx_2\ar[rrdd]\ar@/^1.4pc/@{..}[rrrrrdd]|{\sz{n}} \\
&&&&&&&&&&&&&&&&\\
\bT(n):&\vec{0}\ar[rrrr]\ar[rruu]\ar[rrdd]\ar@/^6.5pc/@{..}[rrrrrrrrrrr]|{\sz{4n+1}}\ar@/^1.4pc/@{..}[rrrrrrrr]&&&&\vx_1\ar[rrrr]\ar@/_1pc/@{..}[rrrrrrr]|{\sz{2n}}&&
&&\vc\ar[rrr]|{\sz{8n}}&&&b_3\ar[r]&b_2\ar[r]&b_1\\
&&&&&&&& \\
&&&\vx_3\ar[rr]\ar@/_0.8pc/@{..}[rrrrrrrrruu]|{\sz{n-1}}&&2\vx_3\ar[rr]\ar@/_1.1pc/@{..}[rrrrrrruu]|{\sz{n}}&&3\vx_3\ar[rruu]\ar@/_1.4pc/@{..}[rrrrruu]|{\sz{n}}\\
}
$$
This uses Proposition~\ref{prop:factorization}.
Applying H\"{u}bner reflections in the vertices $b_3,b_2,b_1$ in this order, we finally obtain a sequence of tilting bundles on $\XX$ whose endomorphism rings are given as follows.
$$
\def\bs{b^\ast}
\xymatrix@C12pt@R15pt{
&&&&&2\vx_2\ar[rr]\ar[ddd]|{\sz{n}}\ar@/^1.3pc/@{..}[rrrrrrrddd]|{\sz{8n^2}}&&3\vx_2\ar[llddd]|{\sz{n}}\ar@/^1.6pc/@{..}[rrrrrddd]|{\sz{8n^2-1}}& \\
&&&\vx_2\ar[rru]\ar[rrdd]|{\sz{n}}\ar@/^2pc/@{..}[rrrrrrrrrdd]|{\sz{8n^2}}& \\
&&&\vx_1\ar[rrd]|{\sz{2n}}\ar@/^1.6pc/@{..}[rrrrrrrrrd]|{\sz{16n^2-1}}&&&&&&&&&&&\\
\sT(n):&\vec{0}\ar[rru]\ar[rruu]\ar[rrrr]|{\sz{4n+1}}\ar[ddrr]\ar@/^10.5pc/@{..}[rrrrrrrrrrr]|{\sz{32n^2+8n+1}}&&&&\bs_1\ar[rr]&&
\bs_2\ar[rr]&&\bs_3\ar[rrr]|{\sz{8n}}&&&\vc\\
&&&&&&&& \\
&&&\vx_3\ar[rrd]\ar[rruu]|{\sz{n-1}}\ar@/_1.3pc/@{..}[rrrrrrrrruu]|{\sz{8n^2-8n}}&&&\\
&&&&&2\vx_3\ar[rr]\ar[uuu]|{\sz{n}}\ar@/_1.3pc/@{..}[rrrrrrruuu]|{\sz{8n^2}}&&3\vx_3\ar[uuull]|{\sz{n}}\ar@/_1.6pc/@{..}[rrrrruuu]|{\sz{8n^2-1}}&\\
}
$$
The degree, rank and slope data for the tilting bundles $\sT(n)$  are collected in the following table.
$$
\def\bs{b^\ast}
\begin{array}{cccccccccccc}
\vec{0}            &\vx_1               &\vx_2               &2\vx_2       &3\vx_2&\vx_3\\[1pt] \hline\\[-8pt]

\frac{8n+24}{16n+1}&\frac{12n+38}{24n+1}&\frac{10n+31}{20n+1}&\frac{12n+38}{24n+1}&\frac{14n+45}{28n+1}&\frac{10n+31}{20n+1}\\
                   &                    &                    &                    &                    &               &\\2\vx_3&3\vx_3&\vc&\bs_1&\bs_2&\bs_3\\[1pt] \hline\\[-8pt]
\frac{12n+44}{24n+1}&\frac{14n+48}{28n+1}&\frac{16n+52}{32n+1}&\frac{128n^2+416n-12}{236n^2+8n}&\frac{128n^2+416n-8}{256n^2+8n}&\frac{128n^2+416n-4}{256n^2+8n}\\
\end{array}
$$
We observe that all the slope sequences converge to $\frac{1}{2}$. Further, all relations for $\sT(n)$ end in the vertex $\vc$. This can
be rephrased as follows: Let $Q(n)$ be the wild quiver, obtained from the quiver of $\sT(n)$ by removing the last vertex $\vc$.
Then $\End(\sT(n))$ is obtained as a one-point extension of the path algebra $kQ(n)$ of $Q(n)$.

\subsection*{Acknowledgements}
The first named author was supported by DGAPA and PAPIIT,
  Universidad Nacional Aut\'onoma de M\'exico. The second named author
  was supported by the Max Planck Institute for Mathematics, Bonn.

\end{document}